\documentclass[11pt]{article}

\usepackage{a4wide, amsmath, amsfonts, amssymb, mathrsfs, color, amsthm, enumitem, stmaryrd}
\usepackage[hyperfootnotes=false]{hyperref} 

\allowdisplaybreaks[3]

\numberwithin{equation}{section}

\newtheorem{theorem}{Theorem}[section]

\newtheorem{corollary}[theorem]{Corollary}
\newtheorem{remark}[theorem]{Remark}
\newtheorem{proposition}[theorem]{Proposition}
\newtheorem{example}[theorem]{Example}

\theoremstyle{definition}
\newtheorem{definition}[theorem]{Definition}

\newcommand{\dd}{\,\mathrm{d}}
\renewcommand{\d}{\mathrm{d}}
\newcommand{\supp}{\mathrm{supp}}
\newcommand{\R}{\mathbb{R}}

\newcommand{\N}{\mathbb{N}}
\newcommand{\Z}{\mathbb{Z}}

\newcommand{\1}{\mathbf{1}}
\renewcommand{\P}{\mathbb{P}}

\renewcommand{\phi}{\varphi}

\newcommand{\s}{\mathfrak{s}}
\newcommand{\vertiii}[1]{{\left\vert\kern-0.25ex\left\vert\kern-0.25ex\left\vert #1 \right\vert\kern-0.25ex\right\vert\kern-0.25ex\right\vert}}

\title{Stochastic Analysis with Modelled Distributions}
\author{Chong Liu, David J. Pr\"omel, and Josef Teichmann}

\begin{document}

\maketitle

\begin{abstract}
  Using a Besov topology on spaces of modelled distributions in the framework of Hairer's regularity structures, we prove the reconstruction theorem on these Besov spaces with negative regularity. The Besov spaces of modelled distributions are shown to be UMD Banach spaces and of martingale type~$2$. As a consequence, this gives access to a rich stochastic integration theory and to existence and uniqueness results for mild solutions of semilinear stochastic partial differential equations in these spaces of modelled distributions and for distribution-valued SDEs. Furthermore, we provide a Fubini type theorem allowing to interchange the order of stochastic integration and reconstruction.
\end{abstract}



\section{Introduction}

Modelled distributions are the spine of Hairer's theory of regularity structures~\cite{Hairer2014}: they constitute a way to describe locally generalized functions of certain degrees of (ir-)regularity by means of functions (``modelled distributions'') taking values in a graded vector space (``regularity structure''), which satisfy certain graded estimates. A good and simple example are H\"older continuous functions, which can be locally described by the coefficients of their Taylor's expansion up to a certain order around each point with respect to polynomials. One of the key insight of the theory of regularity structures is that the solutions of some singular stochastic partial differential equations, like the KPZ equation or the $2$D~parabolic Anderson model, are more suitable described using an enlarged basis of monomials. 

In the abstract setting of regularity structures, the so-called reconstruction operator provides a way to continuously map the modelled distributions\footnote{collections of ``Taylor'' coefficients around each point w.r.t. a set of abstract monomials} to generalized $\R^d$-valued functions on space-time, which is the assertion of the celebrated reconstruction theorem, see Theorem~3.10 in~\cite{Hairer2014}. In the seminal work~\cite{Hairer2014}, the spaces of modelled distributions are equipped with the direct analogues of H\"older norms, which was more than sufficient for the original applications and it is most natural from the point of view of the reconstruction theorem.  

However, with stochastic analysis and especially stochastic integration in mind, the more general Besov and, in particular, Sobolev--Slobodeckij type norms are a more natural choice since these norms provide a suitable geometric structure to have a rich stochastic integration theory at hand, as we will demonstrate in the second part of the present paper. Other motivations to work with specific Besov norms on the spaces of modelled distributions recently arose in the work~\cite{Hairer2018} of Hairer and Labb\'e, where the solution to the multiplicative stochastic heat equation starting from a Dirac mass is constructed, and in the work~\cite{Cannizzaro2017} of Cannizzaro, Friz and Gassiat, where Malliavin calculus is implemented in the context of regularity structures.

It is the goal of the first part of the article to show that the reconstruction theorem still holds in full generality if the spaces of modelled distributions are equipped with Besov norms. While \cite{Hairer2017} already provide the reconstruction theorem for Besov spaces of modelled distributions assuming the regularity parameter to be positive, we complement their result by proving the reconstruction theorem for Besov spaces with negative regularity, see Theorem~\ref{thm:reconstruction negative}. As in the case of Hairer's original reconstruction theorem, these two regimes require different proofs but both can be proven along similar lines of arguments as used by Hairer's original proof on the existence of the reconstruction operator. Let us also mention that the reconstruction theorem was recently obtained by Hensel and Rosati~\cite{Hensel2017} for Triebel--Lizorkin type spaces with positive regularity parameter. A natural application of the reconstruction operator applied to modelled distributions with negative regularity is Lyons--Victoir's extension theorem~\cite{Lyons2007a}, cf.~\cite[Proposition~13.23]{Friz2014}.

The reconstruction operator $\mathcal{R}$ maps modelled distributions to generalized functions in a linear and bounded way with additional continuous dependence on the underlying model. The reconstruction operator can be considered as an abstract integration operation, which depends on the particular regularity structure. It generalizes Young integration~\cite{Young1936} and controlled rough path integration~\cite{Lyons1998,Gubinelli2004}, etc. The main result of this article (Theorem~\ref{thm:fubini theorem}) can be seen as a Fubini type theorem, which asserts for modelled-distribution-valued predictable processes~$H$ and a Brownian motion~$W$ that the order of ``integration'' can be interchanged
\begin{equation*}
  \bigg \langle \mathcal{R} \big({(H \bullet W)} \big) , \psi \bigg \rangle = \bigg( \big \langle \mathcal{R}(H),\psi \big \rangle \bullet W \bigg ) \, 
\end{equation*}
for every test function $\psi$, where $(H \bullet W)$ denotes the stochastic integral of~$H$ w.r.t.~$W$. This Fubini type theorem has a deeper meaning if the Besov space~$\mathcal{D}^\gamma_{p,q}$ of modelled distributions has a geometric structure such that a rich stochastic integration theory is accessible.

There are several approaches to stochastic integration for Banach space valued processes, some of them involve properties of the Banach space like martingale type~$2$ or unconditional martingale difference (UMD). It depends on the purpose in mind, which property is actually needed, but for integrals with respect to Brownian motion martingale type~$2$ or UMD is favorable and both allow for treating stochastic partial differential equations like stochastic evolution equations in Banach spaces, see e.g.~\cite{Brzezniak1995}, \cite{Brzezniak2018} or \cite{vanNeerven2008}. We shall prove here that the Besov space~$\mathcal{D}^\gamma_{p,q}$ of modelled distributions (for $p,q \geq 2$) has indeed the martingale type~$2$ and the UMD property, respectively, see Proposition~\ref{prop:UMD property}. Since this suffices to set up a rich stochastic integration theory as needed for the treatment of stochastic partial differential equations with Brownian drivers like in the books of Da~Prato--Peszat--Zabczyk~\cite{Peszat2007,DaPrato2014}, the results in the second part of the article pave the way to combine the powerful tools of stochastic integration and Hairer's theory of regularity structure in a novel way. As an exemplary applications, we present existence and uniqueness results for mild solutions of semilinear stochastic partial differential equations with multiplicative noise in Besov spaces of modelled distributions and for It\^o stochastic differential equations on spaces of H\"older functions with certain unbounded vector fields. Furthermore, using the reconstruction operator and the Fubini type theorem, we show that semilinear SPDEs in classical Besov spaces can be equivalently considered as semilinear SPDEs in suitable Besov spaces of modelled distributions.

\medskip
\noindent {\bf Organization of the paper:} In Section~\ref{sec:reconstruction} we briefly introduce the necessary elements of regularity structures and prove the reconstruction theorem for Besov spaces of modelled distributions. The Banach space properties of these spaces are established in Section~\ref{sec:stochastic integration} as well as the Fubini type theorem. As applications, we treat semilinear stochastic partial differential equations in Besov spaces of modelled distributions in Section~\ref{sec:SPDEs} and distributions-valued It\^o stochastic differential equations in Section~\ref{sec: distribution valued SDEs}.

\medskip
\noindent {\bf Notation:} 
Throughout the entire paper, we are given a scaling $\s:=(\s_1, \dots, \s_d)\in \N^d$ and we consider the $\s$-scaled ``norm'' $\|x\|_{\s}:= \sup_{i=1,\dots, d} |x_i|^{1/\s_i}$ for $x=(x_1,\dots,x_d)\in \R^d$ and $d\in\mathbb{N}$. The ball in $\mathbb{R}^{d}$, around $x\in\mathbb{R}^{d}$ with radius $R>0$ and with respect to the $\s$-scaled norm is denoted by $B(x,R)$. The zero is included in our notation of natural numbers $\mathbb{N}=\{0,1,2,\ldots\}$ and for a multi-index $k\in\mathbb{N}^{d}$, we set $|k|:=|(k_{1},\dots,k_{d})|:=k_{1}+\cdots+k_{d}$, $|k|_\s:=\s_1 k_{1}+\cdots+ \s_d k_{d}$, and $k!:=k_{1}!\cdots k_{d}!$. 

For two real functions $a,b$ depending on variables $x$ one writes $a\lesssim b$ or $a\lesssim_z b$ if there exists a constant $C(z)>0$ such that $a(x) \leq C(z)\cdot b(x)$ for all $x$, and $a\sim b$ if $a\lesssim b$ and $b\lesssim a$ hold simultaneously. By $\lfloor a\rfloor$ for a number $a\in\mathbb{R}$ we mean $\lfloor a\rfloor:=\sup\{b\in\mathbb{Z}\,:\, b\leq a\}$. 

The space of H\"older continuous functions $\phi\colon\mathbb{R}^{d}\to\mathbb{R}$ of order $r \geq 0$ is denoted by $\mathcal{C}^{r}$, that is, $\phi$ is bounded if $ r = 0 $, H\"older continuous for $ 0 < r \leq 1 $  (which amounts precisely to Lipschitz continuous for $r=1$, the derivative does not necessarily exist). For $ r > 1 $ not an integer the function is $\lfloor r\rfloor $-times continuously differentiable and the derivatives of order $\lfloor r\rfloor$ are H\"older continuous of order $r-\lfloor r\rfloor$. The space $\mathcal{C}^{r}$ is equipped with the norm
\[
  \lVert f \rVert_{C^r} := \sum_{k=0}^{\lfloor r \rfloor} \| D^k f \|_\infty + \1_{r>\lfloor r \rfloor} \| D^{\lfloor r \rfloor} f \|_{r - \lfloor r \rfloor},
\]
where $\lVert \cdot \rVert_\beta$ denotes the $\beta$-H\"older norm for $\beta \in (0,1)$, and $\lVert \cdot \rVert_\infty$ denotes the supremum norm.
For integers $ r > 1 $ the $(r-1)$-th derivative exists and is Lipschitz continuous. If a function $\phi\in\mathcal{C}^{r}$ has compact support, we say $\phi\in\mathcal{C}_{0}^{r}$. Additionally, we use $\phi\in\mathcal{B}^{r}$ if $\phi\in\mathcal{C}_{0}^{r}$ is such that $\|\phi\|_{\mathcal{C}^{r}}\leq 1$ and $\supp\,\phi\subset\mathcal{B}(0,1)$, and $\phi\in\mathcal{B}^{r}_{n}$ for $n\in \N$ if $\phi \in \mathcal{B}^{r}$ and $\phi$ annihilates all polynomials of scaled degree at most $n$. As usual $\mathcal{C}^{\infty}=\mathcal{C}^{\infty}(\mathbb{R}^{d})$ stands for the space of smooth functions $\phi\colon\mathbb{R}^{d}\to\mathbb{R}$ and $\mathcal{C}_{0}^{\infty}$ is the subspace of all smooth functions with compact support. The space $\mathcal{D}^{\prime}=\mathcal{D}^{\prime}(\mathbb{R}^{d})$ is the space of (tempered) distributions, that is, the topological dual of the Schwartz space of rapidly decreasing infinitely differentiable functions.  

The space $L^{p}:=L^{p}(\mathbb{R}^{d},\d x)$, $p\geq1$, is the usual Lebesgue space, that is, the space of functions~$f$ such that $\int_{\mathbb{R}^{d}}|f(x)|^{p}\dd x<\infty$. We also set $L^q_\lambda := L^q((0,1),\lambda^{-1}\d \lambda)$ for $q\geq 1$ and write $L^{p}(\mathbb{R}^{d};B)$ for the $L^p$-space of functions $f\colon \R^d\to B$ where $B$ is Banach space. The notation $\langle f, g\rangle$ is used for the $L^2$-inner product of $f$ and $g$ as well as the evaluation of the distribution $f$ against the test function~$g$. The space $\ell^p$ is the Banach space of all sequences $(x_n)_{n\in \N}$ such that $\sum_{n\in \N}|x_n|^p<\infty$ and the corresponding norm is denoted by $\|\cdot \|_{\ell^p}$.

\section{Reconstruction operator and Besov modelled distributions}\label{sec:reconstruction}

The theory of regularity structures was introduced by M. Hairer in the seminal work~\cite{Hairer2014}. Gentle introductions to this novel theory can be found, for instance, in \cite{Friz2014}, \cite{Hairer2015} or \cite{Chandra2017}. We recall here for the sake of completeness the fundamental objects in suitable generality for the present paper. For the convenience of the reader our notation and definitions are mainly borrowed from \cite{Hairer2014,Hairer2015}. Let us start with the definition of a \emph{regularity structure}, of its \emph{models} and of \emph{modelled distributions}.

\begin{definition}[Definition~2.1 in \cite{Hairer2014}]
  A triplet $\mathcal{T}=(A,T,G)$ is called \textit{regularity structure} if it consists of the following objects:
  \begin{itemize}
    \item An \textit{index set} $A\subset\R$, which is locally finite and bounded from below, with $0\in A$.
    \item A \textit{model space} $T=\bigoplus_{\alpha\in A}T_{\alpha}$, which is a graded vector space with each $T_{\alpha}$ a Banach space and $T_{0}\approx\R$. Its unit vector is denoted by $\mathbf{1}$.
    \item A \textit{structure group} $G$ consisting of linear operators acting on $T$ such that, for every $\Gamma\in G$, every $\alpha\in A$, and every $a\in T_{\alpha}$ it holds
          \begin{align*}
            \Gamma a-a\in\bigoplus_{\beta\in A;\,\beta<\alpha}T_{\beta}.
          \end{align*}
          Moreover, $\Gamma\mathbf{1}=\mathbf{1}$ for every $\Gamma\in G$. 
  \end{itemize}
  For any $\tau\in T$ and $\alpha\in A$ we denote by $\mathcal{Q}_{\alpha}\tau$ the projection of $\tau$ onto $T_{\alpha}$ and set $\|\tau\|_{\alpha}:=\|\mathcal{Q}_{\alpha}\tau\|$.
\end{definition}

The basic idea behind the \emph{model space}~$T$ is to represent abstractly the information describing the ``jet'' or ``local expansion'' of a (generalized) function at any given point, i.e.~we prescribe a certain structure of local expansions of (generalized) functions, which we have in mind. Each $T_{\alpha}$ then corresponds to the ``monomials of degree~$\alpha$'' which are required to describe a (generalized) function locally ``of order~$\alpha$'' and the role of the structure group~$G$ is to translate coefficients from a local expansion around a given point into coefficients for an expansion around another point, such that the (generalized) function does not change. To make this interpretation clearer, we present the abstract polynomials as very simple example of a regularity structure. A more detailed discussion of this example can be found in Section~2.2 in~\cite{Hairer2014} or Section~13.2.1 in~\cite{Friz2014}. Alternatively, the reader might keep in mind the theory of (controlled) rough paths~\cite{Lyons1998,Gubinelli2004} as an example of a regularity structure, see Section~13.2.2 in~\cite{Friz2014}. 

\begin{example}\label{ex:polynomial}
  The polynomial regularity structure $\overline{\mathcal{T}}$ is given by the space of abstract polynomials in $d$ variables. In this case the index set is the set of natural numbers, that is $A=\mathbb{N}$. The model space $T=\mathbb{R}[X_{1},\dots,X_{d}]$ is indeed a graded vector space since it can be written as 
  \begin{equation*}
    T=\bigoplus_{\alpha\in A}T_{\alpha}\quad\text{with}\quad T_{\alpha}:=\mathrm{span}\,\{X^{k}\ :\ |k|=\alpha\}, 
  \end{equation*}
  where $\mathrm{span}\,\{X^{k}\, :\, |k|=\alpha\}$ is the space generated by all monomials of degree~$\alpha$ and $X^{k}:=X^{k_{1}}\dots X^{k_{d}}$. Here we use for simplicity the scaling $\s := (1,\dots,1)$. The canonical group action is $G\sim(\mathbb{R}^{d},+)$ which acts on $T$ via $\Gamma_{h}P(X):=P(X+h\mathbf{1})$ for every $h\in\mathbb{R}^{d}$ and $P(X)\in T$.
\end{example}

In order to associate to each ``abstract'' element in~$T$ a ``concrete'' (generalized) function or distribution on~$\mathbb{R}^{d}$, M. Hairer introduced the concept of \emph{models}. 

\begin{definition}[Definition~2.17 in \cite{Hairer2014}]\label{def:local model}
  Given a regularity structure $\mathcal{T}=(A,T,G)$, a \textit{model $(\Gamma,\Pi)$} on $\R^{d}$ is given by:
  \begin{itemize}
    \item A linear map $\Gamma\colon\R^{d}\times\R^{d}\to G$ such that $\Gamma_{x,y}\Gamma_{y,z}=\Gamma_{x,z}$ for every $x,y,z\in\R^{d}$ and $\Gamma_{x,x}=1$, where $1$ is the identity operator. 
    \item A collection of continuous linear maps $\Pi_{x}\colon T\to\mathcal{D}^{\prime}$ such that $\Pi_{y}=\Pi_{x}\circ\Gamma_{x,y}$ for every $x,y\in\R^{d}$.
  \end{itemize}
  Furthermore, for every compact set $\mathcal{K}\subset\mathbb{R}^{d}$ and for every constant $\gamma>0$, there exists a constant $C_{\gamma,\mathcal{K}}>0$ such that the bounds
  \begin{equation}\label{eq:local model}
    |\langle\Pi_{x}\tau,\phi_{x}^{\lambda}\rangle|\leq C_{\gamma,\mathcal{K}}\lambda^{\alpha}\|\tau\|_{\alpha}\quad\text{and}\quad\|\Gamma_{x,y}\tau\|_{\beta}\leq C_{\gamma,\mathcal{K}}\|x-y\|_{\s}^{\alpha-\beta}\|\tau\|_{\alpha}
  \end{equation}
  hold uniform over $(x,y)\in\mathcal{K}\times\mathcal{K}$, $\lambda\in(0,1]$, $\tau\in T_{\alpha}$ for $\alpha\leq\gamma$ and $\beta\leq\alpha$, and for all $\phi\in\mathcal{B}^{r}$ with $r>|\inf A|$. Additionally, we denote by $\|\Pi\|_{\gamma,\mathcal{K}}$ and $\|\Gamma\|_{\gamma,\mathcal{K}}$ the smallest constant for which the first and second inequality in~\eqref{eq:local model} holds, respectively.
\end{definition}

To illustrate the definition of a model, let us come back to Example~\ref{ex:polynomial}.

\begin{example}\label{ex:polymodel}
  Given the polynomial regularity structure $\overline{\mathcal{T}}=(A,T,G)$ from Example~\ref{ex:polynomial}, a corresponding model $(\Gamma,\Pi)$ can be defined by the concrete polynomials on $\mathbb{R}^{d}$. More precisely, the model $(\Gamma,\Pi)$ is given by the action 
  \begin{equation*}
    \Gamma_{x,y}P(X):=P(X+(x-y))\quad\text{and}\quad\Pi_{x}P(X):=P(\cdot-x),
  \end{equation*}
  for $X\in T$ and $x,y\in\mathbb{R}^{d}$.
\end{example}

The functions which can be described by the polynomial regularity structure and the model as introduced in Example~\ref{ex:polynomial} and \ref{ex:polymodel}, are the H\"older continuous functions. Indeed, take a function $f\in\mathcal{C}^{\gamma}$ for some~$\gamma>0$. Using the Taylor expansion of order $\lfloor\gamma\rfloor$, one can associate to $f$ a map $\hat{f}$ with values in $T$ via 
\begin{equation*}
  \hat{f}\colon\mathbb{R}^{d}\to\bigoplus_{\substack{\alpha\in A,\, \alpha<\gamma}}T_{\alpha}\subset T\quad\text{with }\quad\hat{f}(x):=\sum_{k\in\mathbb{N},\,|k|<\gamma}\frac{\partial^{k}f(x)}{k!}X^{k}.
\end{equation*}
Equipping a suitable subspace of functions of the form
\begin{equation*}
  \hat{f}\colon\mathbb{R}^{d}\to \bigoplus_{\substack{\alpha\in A,\, \alpha<\gamma}}T_{\alpha}
\end{equation*}
with the right topology, the map $f\mapsto\hat{f}$ turns out to be a one-to-one correspondence as proven in Lemma~2.12 in \cite{Hairer2014}.\smallskip

In the general context of regularity structures the ``H\"older continuous'' functions relative to a given model are the so-called ``modelled distributions''. This class of distributions locally ``looks like'' the distributions in the model. The precise definition reads as follows.

\begin{definition}[Definition~3.1 in \cite{Hairer2014}]\label{def:Holder modelled distributions}
  Let $\gamma\in\R$. The space of \textit{modelled distributions}~$\mathcal{D}^{\gamma}$ is given by all functions $f\colon\R^{d}\to T_{\gamma}^{-}$ such that for every compact set $\mathcal{K}\subset\R^{d}$ one has 
  \begin{equation*}
    \interleave f\interleave_{\gamma,\mathcal{K}}:=\sup_{x\in\mathcal{K}}\sup_{\alpha\in A_{\gamma}}\|f(x)\|_{\alpha}+\sup_{\substack{x,y\in\mathcal{K}\\ \|x-y\|_\s\leq1}}\sup_{\alpha\in A_{\gamma}}\frac{\|f(x)-\Gamma_{x,y}f(y)\|_{\alpha}}{\|x-y\|_\s^{\gamma-\alpha}}<\infty.
  \end{equation*}
  Here, we used the notation $T_{\gamma}^{-}:=\bigoplus_{\alpha\in A_{\gamma}}T_{\alpha}$, where one denotes $A_{\gamma}:=\{\alpha\in A\,:\,\alpha<\gamma\}$.
\end{definition}

Maybe the most fundamental result in Hairer's theory of regularity structures is the reconstruction theorem (Theorem~3.10 in \cite{Hairer2014}): for every $f\in\mathcal{D}^{\gamma}$ with $\gamma\in\R$ there exists a distribution $\mathcal{R}f$ with some (possibly negative) H\"older regularity on $\mathbb{R}^{d}$ such that $\mathcal{R}f$ ``looks like $\Pi_{x}f(x)$ near~$x$'' for every $x\in\mathbb{R}^{d}$. In other words, it is always possible to obtain from an abstract map $f\in\mathcal{D}^{\gamma}$ a concrete distribution $\mathcal{R}f$, which locally looks in some sense like~$f$.

\subsection{Reconstruction theorem for Besov spaces with negative regularity}

As already discussed in the Introduction, from a probabilist's point of view it seems more desirable to work with $L^{p}$-type norms instead of $L^{\infty}$-norms (as used to measure H\"older regularity as in Definition~\ref{def:Holder modelled distributions}) since this has the great advantage to give access to strong and highly developed techniques as stochastic integration. Therefore, we would like to work with a generalized version of the space of modelled distributions which is the analogue to classical Besov spaces. In this new setting it seems to be very natural and convenient to introduce models possessing global bounds instead of the local ones required in Definition~\ref{def:local model}. Therefore, following \cite{Hairer2017} we use a slight modified definition of models but we will come back to the original framework of regularity structures in Subsection~\ref{subsec:reconstruction for local models}.

\begin{definition}[Definition~2.8 in \cite{Hairer2017}]\label{def:global model}
  Given a regularity structure $\mathcal{T}=(A,T,G)$, a \textit{model $(\Gamma,\Pi)$} on $\R^{d}$ is given by:
  \begin{itemize}
    \item A linear map $\Gamma\colon\R^{d}\times\R^{d}\to G$ such that $\Gamma_{x,y}\Gamma_{y,z}=\Gamma_{x,z}$ for every $x,y,z\in\R^{d}$ and $\Gamma_{x,x}=1$, where $1$ is the identity operator. 
    \item A collection of continuous linear maps $\Pi_{x}\colon T\to\mathcal{D}^{\prime}$ such that $\Pi_{y}=\Pi_{x}\circ\Gamma_{x,y}$ for every $x,y\in\R^{d}$.
  \end{itemize}
  Furthermore, for every constant~$\gamma>0$, there exists a constant $C_{\gamma}>0$ such that the bounds
  \begin{equation}\label{eq:global model}
    |\langle\Pi_{x}\tau,\phi_{x}^{\lambda}\rangle|\leq C_{\gamma}\lambda^{\alpha}\|\tau\|_{\alpha}
    \quad\text{and}\quad
    \|\Gamma_{x,y}\tau\|_{\beta}\leq C_{\gamma}\|x-y\|_{\s}^{\alpha-\beta}\|\tau\|_{\alpha}
  \end{equation}
  hold uniform over $(x,y)\in\R^d\times\R^d$, $\lambda\in(0,1]$, $\tau\in T_{\alpha}$ for $\alpha\leq\gamma$ and $\beta\leq\alpha$, and for all $\phi\in\mathcal{B}^{r}$ with $r>|\inf A|$. Additionally, we denote by $\|\Pi\|:=\|\Pi\|_{\gamma}$ and $\|\Gamma\|:=\|\Gamma\|_{\gamma}$ the smallest constant for which the first and second inequality in~\eqref{eq:global model} holds, respectively. In the following, we drop the dependence on $\gamma$ whenever it is clear from the context. 
  
  Given two models $(\Gamma,\Pi)$ and $(\overline{\Gamma},\overline{\Pi})$ for the same regularity structure $\mathcal{T}=(A,T,G)$, a natural pseudo-metric is induced by
  \begin{align*}
    &\|\Pi-\overline{\Pi}\|:= \sup_{x\in\R^d}\sup_{\phi\in\mathcal{B}^{r}} \sup_{\lambda\in(0,1]} \sup_{\alpha \in A_{\gamma}} \sup_{\tau\in T_{\alpha}} \frac{ |\langle\Pi_{x}\tau-\overline{\Pi}_{x}\tau,\phi_{x}^{\lambda}\rangle|}{\lambda^{\alpha}\|\tau\|_{\alpha}},\\
    &\|\Gamma-\overline{\Gamma}\|:= \sup_{(x,y)\in\R^d\times\R^d}\sup_{\phi\in\mathcal{B}^{r}} \sup_{\lambda\in(0,1]} \sup_{\beta\leq\alpha \in A_{\gamma}} \sup_{\tau\in T_{\alpha}}\frac{ \|\Gamma_{x,y}\tau-\overline{\Gamma}_{x,y}\tau\|_{\beta}}{\|x-y\|_{\s}^{\alpha-\beta}\|\tau\|_{\alpha}},
  \end{align*}
  where we recall $A_\gamma := \{\alpha \in A \,:\, \alpha < \gamma \}$.
\end{definition}

For the rest of the subsection, we fix an arbitrary regularity structure $\mathcal{T}=(A,T,G)$ with an associated model $(\Gamma,\Pi)$ satisfying global bounds in the sense of Definition~\ref{def:global model}. Let us recall the scaling $\s=(\s_1,\dots,\s_d)$ on $\R^d$ and let $\mathcal{K} \subset \R^d$ be a Borel measurable set. For a measurable function $f\colon\R^{d}\to T_{\gamma}^{-}$, $\gamma\in \R$, $\alpha\in \mathcal{A}_\gamma$ and $p,q\in [1,\infty)$, we define 
\begin{equation*}
  \| |f|_{\alpha} \|_{L^p(\mathcal{K})} := \bigg(\int_{\mathcal{K}}|f(x)|_{\alpha}^{p}\dd x\bigg)^{\frac{1}{p}}, \quad 
  \| |f|_{\alpha} \|_{L^p}:= \|  |f|_{\alpha} \|_{L^p(\R^d)} 
\end{equation*}
and introduce the Besov norm 
\begin{align*} 
\begin{split}
  \interleave f \interleave_{\gamma,p,q,\mathcal{K}} 
  := & \sum_{\alpha \in A_\gamma}\| |f(x)|_{\alpha} \|_{L^p(\overline{\mathcal{K}})}\\
  & + \sum_{\alpha \in A_\gamma}\Big(\int_{h \in B(0,1)} \left\|\frac{|f(x+h) - \Gamma_{x+h,x}f(x)|_\alpha}{\|h\|^{\gamma - \alpha}_{\mathfrak{s}}}\right\|_{L^p(\mathcal{K})}^q\frac{\d h}{\|h\|^{|\mathfrak{s}|}_{\mathfrak{s}}}\Big)^{\frac{1}{q}}, 
\end{split}
\end{align*}
where $\overline{\mathcal{K}}$ stands for the $1$-fatting of the set $\mathcal{K}$ and where we used shortened notation by writing $\sum_{\alpha <\gamma}$ meaning $\sum_{\alpha \in A_\gamma}$. The norm $\interleave \cdot \interleave_{\gamma,p,q,\R^d}$ can be considered as the analogue to  classical Besov norms based on their definition using integrals, cf. for instance~\cite{Simon1990} or \cite{Triebel2010}.

\begin{definition}\label{def:Besov modelled distribution}
  Let $\gamma\in \R$ and $p,q\in [1,\infty)$. The \textit{Besov space}~$\mathcal{D}_{p,q}^{\gamma}$ consists of all functions $f\colon\R^{d}\to T_{\gamma}^{-}$ such that 
  $$
    \interleave f\interleave_{\gamma,p,q,\R^d}:=\interleave f\interleave_{\gamma,p,q}<\infty.
  $$
  Further, we write~$\mathcal{D}_{p,q}^{\gamma} (\mathcal{K})$ for the space of all functions $f\colon\overline{\mathcal{K}}\to T_{\gamma}^{-}$ such that $\interleave f\interleave_{\gamma,p,q,\mathcal{K}}<\infty$, for a compact set $\mathcal{K}\subset\R^d$.
\end{definition}

\begin{remark}
  We would like to point out that the general Besov spaces~$\mathcal{D}_{p,q}^{\gamma}$ were first defined by Hairer and Labb\'e~\cite{Hairer2017}.
  In an early version of the present article only the special case of Sobolev--Slobodeckij spaces (also called fractional Sobolev spaces) of modelled distributions were introduced, which appears to be sufficient in many situations for stochastic integration. The Sobolev--Slobodeckij spaces correspond to the Besov spaces~$\mathcal{D}_{p,p}^{\gamma}$. Another different $L^p$-counterpart of the original space~$\mathcal{D}^\gamma$ of modelled distributions were considered in \cite{Hairer2018}, which is related to Nikolskii spaces or in other words to Besov spaces~$\mathcal{D}^\gamma_{p,\infty}$, see Definition~2.9 in~\cite{Hairer2018}.
\end{remark}

\begin{remark}
  The Besov spaces~$\mathcal{D}_{p,q}^{\gamma}$ enjoy similar embedding properties as the well-known embedding theorems for classical Besov spaces.  In particular, the Besov spaces are nested in their regularity parameter, that is  $\mathcal{D}_{p,q}^{\gamma_{1}}\subset\mathcal{D}_{p,q}^{\gamma_{2}}$ for $\gamma_{1}\geq\gamma_{2}$ and $p,q\in[1,\infty)$. 
  For more sophisticated embedding results we refer to Section~4 in \cite{Hairer2017}.
\end{remark}

Because Hairer's reconstruction operator maps modelled distributions possessing some H\"older regularity to generalized functions again possessing certain H\"older regularity, one expects a similar result also for modelled distributions with Besov regularity. Here we focus on the Besov spaces with negative regularity since this suffices for our reconstruction theorem, see Theorem~\ref{thm:reconstruction negative} below. For the general definition we refer to Definition~2.1 in \cite{Hairer2017} and to introductory books about Besov spaces as, for instance, to \cite{Triebel2010} or \cite{Bahouri2011}. 

\begin{definition}[Definition~2.1 in \cite{Hairer2017}]
  Let $\alpha <0 $, $p,q\in [1,\infty)$ and $r\in \N$ such that $r>|\alpha|$. The \textit{Besov space}~$\mathcal{B}^{\alpha}_{p,q}:=\mathcal{B}^{\alpha}_{p,q}(\R^d)$ is the space of all distributions~$\xi$ on $\R^d$ such that 
  \begin{equation*}
    \bigg\| \Big\| \sup_{\eta\in \mathcal{B}^r(\R^d)} \frac{|\langle \xi , \eta^\lambda_x \rangle|}{\lambda^\alpha}  \Big\|_{L^p} \bigg\|_{L^q_\lambda} <\infty,
  \end{equation*}
  where 
  $$
    \eta^\lambda_x(y):= \lambda^{-|\s|}\eta (\lambda^{-\s_1}(y_1-x_1),\dots,\lambda^{-\s_d}(y_d-x_d))
  $$
  for $\lambda\in (0,1]$, $x=(x_1,\dots,x_d)\in \R^d$ and $y=(y_1,\dots,y_d)\in \R^d$.
\end{definition}

The next theorem presents Hairer's celebrated reconstruction theorem for the Besov space~$\mathcal{D}_{p,q}^{\gamma}$ with negative regularity~$\gamma$. The reconstruction theorem for the Besov spaces~$\mathcal{D}_{p,q}^{\gamma}$ with positive regularity~$\gamma$ can be found in Theorem~3.1 in \cite{Hairer2017}. While the reconstruction operator is unique in the later case, this uniqueness is lost for the reconstruction operator acting on modelled distributions with negative regularity.   

\begin{theorem}\label{thm:reconstruction negative} 
  Suppose that $\alpha := \min A < \gamma < 0 $, and $f\in \mathcal{D}^\gamma_{p,q}$. If $q = \infty $, let $\bar{\alpha} = \alpha$, otherwise take $\bar{\alpha} < \alpha$. Then, there exists a continuous linear operator $\mathcal{R} \colon \mathcal{D}^\gamma_{p,q} \rightarrow \mathcal{B}^{\bar{\alpha}}_{p,q}$ such that 
  \begin{equation}\label{eq:resonstruction bound}
    \left\lVert \Big\| \sup_{\eta \in \mathcal{B}^r} \frac{|\langle \mathcal{R}f - \Pi_xf(x), \eta^\lambda_x \rangle|}{\lambda^\gamma} \Big\|_{L^p} \right\rVert_{L^q_\lambda} \lesssim \|\Pi\|(1 + \|\Gamma\|)\vertiii{f}_{\gamma,p,q},
  \end{equation}
  uniformly over all $f \in \mathcal{D}^\gamma_{p,q}$ and all models $(\Pi, \Gamma)$ in the sense of Definition~\ref{def:global model}.
  
  Furthermore, let $(\overline{\Pi},\overline{\Gamma})$ be another model in the sense of Definition~\ref{def:global model} for $\mathcal{T}=(A,T,G)$ and denote by $\overline{\mathcal{D}}^\gamma_{p,q}$ the Besov space of modelled distributions w.r.t. $(\overline{\Pi},\overline{\Gamma})$. Then, there exist continuous maps $\mathcal{R}\colon\mathcal{D}^\gamma_{p,q} \rightarrow \mathcal{B}^{\bar{\alpha}}_{p,q}$ and $\overline{\mathcal{R}}\colon\overline{\mathcal{D}}^\gamma_{p,q} \rightarrow \mathcal{B}^{\bar{\alpha}}_{p,q}$ which satisfy the reconstruction bound~\eqref{eq:resonstruction bound} for $(\Pi,\Gamma)$ and $(\overline{\Pi},\overline{\Gamma})$, respectively, and 
  \begin{align}\label{eq:reconstruction bound for two models}
  \begin{split}
    \left\lVert \Big\| \sup_{\eta \in \mathcal{B}^r} \frac{|\langle \mathcal{R}f - \overline{\mathcal{R}}g- \Pi_xf(x) + \overline{\Pi}_xg(x), \eta^\lambda_x \rangle|}{\lambda^\gamma} \Big\|_{L^p} \right\rVert_{L^q_\lambda} &\\
    &\hspace{-3.3cm} \lesssim_{\Pi,\Gamma,\overline{\Pi},\overline{\Gamma}}\vertiii{f;g}_{\gamma,p,q} + \vertiii{g}_{\gamma,p,q}(\| \Pi - \overline{\Pi}\|+\|\Gamma - \overline{\Gamma}\|)
  \end{split}
  \end{align}
  for $f\in\mathcal{D}^\gamma_{p,q}$  and $g\in\overline{\mathcal{D}}^\gamma_{p,q}$, where 
  \begin{align*}
    \vertiii{f;g}_{\gamma,p,q} &=\max_{\zeta \in \mathcal{A}_\gamma} \Big(\Big\||f(x) - g(x)|_\zeta \Big\|_{L^p} \\ 
    &\hspace{0.2cm}+\Big(\int_{B(0,1)} \bigg\|\frac{|f(x+h) - g(x+h) - \Gamma_{x+h,x}f(x) 
  	 +\overline{\Gamma}_{x+h,x}g(x)|_\zeta}{\|h\|_{\s}^{\gamma - \zeta}} \Big\|_{L^p(\d x)}^q\,\frac{\d h}{\|h\|_{\s}^{|\s|}}\Big)^{\frac{1}{q}}\bigg).
  \end{align*}
\end{theorem}

The proof works similarly to the one of the original reconstruction theorem (Theorem~2.10 in \cite{Hairer2015}). However, the relevant estimates need to be generalized to the new $L^p$-setting provided by the Besov space~$\mathcal{D}_{p,q}^{\gamma}$, which, in particular, requires a new/modified definition of the reconstruction operator, cf.~\eqref{eq:reconstruction operator}. Before proving Theorem~\ref{thm:reconstruction negative}, some preliminary discussion is in order: \smallskip

First, we would like to remark that given $f \in \mathcal{D}^{\gamma}_{p,q}$, for any $C>0$ and $\zeta \in A_\gamma$ it holds that 
$$
  \bigg(\int_{h \in B(0,C)} \Big\|\frac{|f(x+h) - \Gamma_{x+h,x}f(x)|_\zeta}{\|h\|_{\s}^{\gamma - \zeta}} \Big\|_{L^p(\d x)}^q\,\frac{\d h}{\|h\|_{\s}^{|\s|}}\bigg)^{\frac{1}{q}} 
  \lesssim (1 + \|\Gamma\|)\vertiii{f}_{\gamma,p,q}.
$$
Indeed, if $C \in (0,1]$, then the above inequality is trivial due to the definition of $\vertiii{f}_{\gamma,p,q}$. Now suppose that $C = 2$. Then we note that
\begin{align*}
  \bigg(\int_{h \in B(0,2)} \Big\| & \frac{|f(x+h) - \Gamma_{x+h,x}f(x)|_\zeta}{\|h\|_{\s}^{\gamma - \zeta}} \Big\|_{L^p(\d x)}^q\,\frac{\d h}{\|h\|_{\s}^{|\s|}}\bigg)^{\frac{1}{q}}\\
  & \qquad \lesssim \bigg(\int_{h \in B(0,1)} \Big\|\frac{|f(x+2h) - \Gamma_{x+2h,x}f(x)|_\zeta}{\|h\|_{\s}^{\gamma - \zeta}} \Big\|_{L^p(\d x)}^q\,\frac{\d h}{\|h\|_{\s}^{|\s|}}\bigg)^{\frac{1}{q}}.
\end{align*}
Since  
\begin{align*}
  |f(x+2h) -& \Gamma_{x+2h,x}f(x)|_\zeta \\
  &\leq |f(x+2h) - \Gamma_{x+2h,x+h}f(x+h)|_\zeta + |\Gamma_{x+2h,x+h}(f(x+h) - \Gamma_{x+h,x}f(x))|_\zeta,
\end{align*}
the triangle inequality yields that
\begin{align*}
  & \hspace{-1cm}\bigg(\int_{h \in B(0,1)} \Big\|\frac{|f(x+2h) - \Gamma_{x+2h,x}f(x)|_\zeta}{\|h\|_{\s}^{\gamma - \zeta}} \Big\|_{L^p(\d x)}^q\,\frac{\d h}{\|h\|_{\s}^{|\s|}}\bigg)^{\frac{1}{q}} \\
  \lesssim & \bigg(\int_{h \in B(0,1)} \Big\|\frac{|f(x+2h) - \Gamma_{x+2h,x+h}f(x+h)|_\zeta}{\|h\|_{\s}^{\gamma - \zeta}} \Big\|_{L^p(\d x)}^q\,\frac{\d h}{\|h\|_{\s}^{|\s|}}\bigg)^{\frac{1}{q}} \\
  &+ \bigg(\int_{h \in B(0,1)} \Big\|\frac{| \Gamma_{x+2h,x+h}(f(x+h) - \Gamma_{x+h,x}f(x))|_\zeta}{\|h\|_{\s}^{\gamma - \zeta}} \Big\|_{L^p(\d x)}^q\,\frac{\d h}{\|h\|_{\s}^{|\s|}}\bigg)^{\frac{1}{q}}.
\end{align*}
Clearly, the first term in the right hand side of the above inequality is bounded by $\vertiii{f}_{\gamma,p,q}$. On the other hand, since
$$
  |\Gamma_{x+2h,x+h}(f(x+h) - \Gamma_{x+h,x}f(x)) |_\zeta \leq \|\Gamma\|\sum_{\gamma >\beta \ge \zeta}|f(x+h) - \Gamma_{x+h,x}f(x)|_\beta\|h\|_{\s}^{\beta - \zeta},
$$
the second term in the right hand side of the above inequality is bounded by 
$$
  \|\Gamma\|\sum_{ \gamma > \beta \ge \zeta}\bigg(\int_{h \in B(0,1)} \Big\|\frac{|f(x+h) - \Gamma_{x+h,x}f(x)|_\beta}{\|h\|_{\s}^{\gamma - \beta}} \Big\|_{L^p(\d x)}^q\,\frac{\d h}{\|h\|_{\s}^{|\s|}}\bigg)^{\frac{1}{q}} \lesssim \|\Gamma\|\vertiii{f}_{\gamma,p,q}.
$$
Hence, summing both terms up, we obtain the desired estimate for $C \in (1,2]$. Then we can easily extend this result to any $C > 2$.\smallskip

Secondly, we need some elements of wavelet analysis for the proof of Theorem~\ref{thm:reconstruction negative}. For more detailed discussions we refer to Section~3.1 in \cite{Hairer2014} and the works of Meyer~\cite{Meyer1992} and of Daubechies~\cite{Daubechies1988}. 

Let $r>0$ be a finite real number. We consider a wavelet basis associated to a scaling function $\phi\colon\mathbb{R}\to\mathbb{R}$ with the following four properties: 
\begin{enumerate}[label=(\roman*)]
  \item The function $\phi$ is in $\mathcal{C}_{0}^{r}$. 
  \item For every polynomial $P$ of degree at most~$r$, one has 
        $$
          \sum_{y\in\mathbb{Z}}\int_{\R} P(z)\phi(z-y) \dd z \,\phi(x-y) =P(x), \quad x\in\mathbb{R}.
        $$
  \item For every $y\in\mathbb{Z}^{d}$ one has $\int_{\mathbb{R}}\phi(x)\phi(x-y)\dd x=\delta_{y,0}$.
  \item There exist coefficients $(a_{k})_{k\in\mathbb{Z}}$ with only finitely many non-zero values such that 
        \begin{equation*}
          \phi(x)=\sum_{k\in\mathbb{Z}}a_{k}\phi(2x-k),\quad x\in \R. 
        \end{equation*}
\end{enumerate}
The existence of such a function $\phi$ can be found in Theorem~13.25 in~\cite{Friz2014} and was originally ensured by Daubechies~\cite{Daubechies1988}. We define 
\begin{equation*}
  \phi_{x}^{n}(y):= \prod_{i=1}^d 2^{\frac{n\s_i}{2}} \phi(2^{n\s_i}(y_i-x_i))
\end{equation*}
for $x=(x_1,\dots,x_d),y=(y_1,\dots,y_d)\in\R^d$ and an $\s$-scaled grid of mesh size~$2^{-n}$ by
\begin{equation*}
  \Lambda_n := \Big \{ (2^{-n \s_1}k_1, \dots, 2^{-n\s_d}k_d)\,:\, k_i \in \Z ,\, i=1,\dots,d\Big\}.
\end{equation*} 
The linear span of $(\phi_{x}^{n})_{x\in\Lambda^{n}}$ is denoted by $V_{n}\subset\mathcal{C}^{r}$ and the $L^{2}$-orthogonal complement of $V_{n-1}$ in $V_{n}$ is denoted by $\hat{V}_{n}$. The subspaces~$\hat{V}_{n}$ can be likewise described than the subspaces~$V_{n}$. Indeed, it is a standard fact coming from wavelet analysis~\cite{Meyer1992}: there exists a finite set~$\Psi$ of compactly supported functions in $\mathcal{C}^r$, that annihilate all polynomials of degree at most~$r$, and such that for every $n\geq 0$,
the set 
\begin{equation*}
  \{\phi_{x}^{n}\,:\, x\in\Lambda_{n}\}\cup\{ \psi_{x}^{m}\,:\,\psi \in \Psi,\, x\in\Lambda_{m},\, m\geq n\}
\end{equation*}
constitute an orthonormal basis of~$L^{2}$. Notice that the subspace $\hat{V}_{n+1}\subset L^2$ is generated by $\{ \psi_{x}^{n}\,:\,\psi \in \Psi,\, x\in\Lambda_{n}\}$.

\begin{proof}[Proof of Theorem~\ref{thm:reconstruction negative}] 
  Let $r\in \N$ such that $r > |\alpha|$ and assume that $\varphi$ and $\psi \in \Psi$ are the father and mother wavelet(s) in $\mathcal{C}^r_0$, respectively, with the above discussed properties. 
  
  In view of \cite[(3.38)]{Hairer2014}, a natural choice for $\mathcal{R}f$ is 
  \begin{equation}\label{eq:reconstruction operator}
    \mathcal{R} f := \sum_{n\in \N} \sum_{x\in \Lambda_n} \sum_{\psi \in \Psi}  \langle \Pi_x \overline{f}^n (x), \psi^n_x \rangle   \psi^n_x + \sum_{x \in \Lambda_0}\langle \Pi_x \overline{f}^0 (x), \varphi^0_x   \rangle   \varphi^0_x,
  \end{equation}
  where $\overline{f}^n$ is given by
  \begin{equation*}
    \overline{f}^n (x):= \int_{B(x,2^{-n})} 2^{n |\s|} \Gamma_{x,y} f(y)\dd y, \quad x \in \Lambda_n,
  \end{equation*}
  cf.~\cite[(2.8)]{Hairer2017}.\smallskip
  
  \textit{Step~1:} Let us first verify that the so-defined $\mathcal{R}f$ belongs to $\mathcal{B}^{\bar{\alpha}}_{p,q}$ for any $\bar{\alpha} < \alpha$. We set for every $n \geq 0$, $x \in \Lambda_n$ and $\psi \in \Psi$ a real number
  $$
    a^{n,\psi}_x := \langle \mathcal{R}f, \psi^n_x \rangle = \langle \Pi_x\overline{f}^n(x), \psi^n_x \rangle,
  $$
  and for $x \in \Lambda_0$, $b^0_x := \langle \mathcal{R}f, \varphi_x^0 \rangle = \langle \Pi_x\overline{f}^0(x), \varphi_x^0 \rangle$. Invoking \cite[Proposition~2.4]{Hairer2017}, it suffices to show that 
  $$
    \left\lVert \Big\|\frac{a^{n,\psi}_x}{2^{-n\frac{|\s|}{2}- n\bar{\alpha}}}\Big\|_{\ell^p_n} \right\rVert_{\ell^q} < \infty 
    \quad\text{and}\quad
    \Big\| b^0_x \Big\|_{\ell^p_0} < \infty,
  $$
  where $\ell^p_n$ stand for the Banach space of all sequences $u(x)$, $x\in \Lambda_n$, such that 
  \begin{equation*}
    \| u(x) \|_{\ell^p_n}:= \Big (\sum_{x\in \Lambda_n} 2^{-n |\s|} |u(x)|^p \Big)^{\frac{1}{p}}<\infty.
  \end{equation*}
  To this end, we remark that by the construction of $\overline{f}^n$
  $$
    |a^{n,\psi}_x|  \leq \int_{B(x,2^{-n})} 2^{n|\s|}|\langle \Pi_x\Gamma_{x,y}f(y), \psi^n_x \rangle| \dd y.
  $$
  Then, since $\psi^n_x = 2^{-n\frac{|\s|}{2}}\psi^{2^{-n}}_x$, from the definition of model we deduce that
  \begin{align*}
    \Big|\langle \Pi_x\Gamma_{x,y}f(y), \psi^n_x \rangle \Big| &\lesssim \| \Pi \|\sum_{\zeta \in A_\gamma} |\Gamma_{x,y}f(y) |_\zeta \, 2^{-n\zeta-n \frac{|\s|}{2}} \\
    &\lesssim \| \Pi \|\| \Gamma \|\sum_{\zeta \in A_\gamma}\sum_{\gamma >\beta \geq \zeta} | f(y) |_\beta \, 2^{-n\beta - n\frac{|\s|}{2}}
  \end{align*}
  uniformly over all $n \geq 0$, all $x \in \Lambda_n$ and all $y \in B(x,2^{-n})$. As a consequence, we get 
  $$
    \Big|\frac{a^{n,\psi}_x}{2^{-n\frac{|\s|}{2}- n\bar{\alpha}}}\Big| \lesssim \sum_{\zeta \in A_\gamma}\sum_{\gamma >\beta \geq \zeta}\int_{B(x,2^{-n})} |f(y)|_\beta \, 2^{n|\s|}2^{n(\bar{\alpha} - \beta)} \dd y,
  $$
  which in turn implies that
  \begin{align*}
    \Big\|\frac{a^{n,\psi}_x}{2^{-n\frac{|\s|}{2}- n\bar{\alpha}}}\Big\|_{\ell^p_n}  &= \Big( \sum_{x \in \Lambda_n}2^{-n|\s|}\Big|\frac{a^{n,\psi}_x}{2^{-n\frac{|\s|}{2}- n\bar{\alpha}}}\Big|^p \Big)^{\frac{1}{p}} \\
    &\lesssim \sum_{\beta \in A_\gamma}\Big( \sum_{x \in \Lambda_n} 2^{-n|\s|}\Big(\int_{B(x,2^{-n})} |f(y) |_\beta 2^{n|\s|} \,2^{n(\bar{\alpha} - \beta)} \dd y \Big)^p  \Big)^{\frac{1}{p}} \\
    &\lesssim \sum_{\beta \in A_\gamma}\Big( \sum_{x \in \Lambda_n} 2^{-n|\s|}\int_{B(x,2^{-n})} \Big(|f(y)|_\beta \, 2^{n(\bar{\alpha} - \beta)}\Big)^p 2^{n|\s|} \dd y  \Big)^{\frac{1}{p}} \\
    &\lesssim \sum_{\beta \in A_\gamma}\Big( \int_{\R^d} | f(y)|^p_\beta \dd y\Big)^{\frac{1}{p}}2^{n(\bar{\alpha} - \beta)},
  \end{align*}
  where we used Jensen's inequality in the third line for the finite measure $2^{n|\s|}\dd y |_{B(x,2^{-n})}$ and the convex function $z \mapsto z^p$. Since $\alpha = \min A$ and $\bar{\alpha} < \alpha$, for all $\beta \in A_\gamma$, one has $\bar{\alpha} - \beta < 0$ and therefore $\sum_{n \geq 0}2^{n(\bar{\alpha} - \beta)} \lesssim 1$. It follows that
  \begin{align*}
    \left\lVert \Big\|\frac{a^{n,\psi}_x}{2^{-n\frac{|\s|}{2}- n\bar{\alpha}}}\Big\|_{\ell^p_n} \right\rVert_{\ell^q} &\lesssim \Big(\sum_{n \geq 0} \Big(\sum_{\beta \in A_\gamma} \Big\| | f(y) |_\beta \Big\|_{L^p} 2^{n(\bar{\alpha}-\beta)}  \Big)^q \Big)^{\frac{1}{q}} \\
    & \lesssim \sum_{\beta \in \mathcal{A}_\gamma} \Big\| | f(y) |_\beta \Big\|_{L^p}\bigg(\sum_{n \geq 0}2^{n(\bar{\alpha} - \beta)q}\bigg)^\frac{1}{q}  \\
    &\lesssim \vertiii{f}_{\gamma,p,q} < \infty,
  \end{align*}
  as claimed. Similarly we can also prove that $\| b^0_x \|_{\ell^p_0} \lesssim \vertiii{f}_{\gamma,p,q} < \infty$. 
  
  In the particular case $q = \infty$, we have 
  $$
   \left\lVert \Big\|\frac{a^{n,\psi}_x}{2^{-n\frac{|\s|}{2}- n\alpha}}\Big\|_{\ell^p_n} \right\rVert_{\ell^\infty} \lesssim \sup_{n \ge 0} \sum_{\beta \in A_\gamma}\Big( \int_{\R^d} | f(y)|^p_\beta \dd y\Big)^{\frac{1}{p}}2^{n(\alpha - \beta)} \lesssim \vertiii{f}_{\gamma,p,q}
  $$
  due to the relation that $\alpha - \beta \le 0$ for all $\beta \in A_\gamma$. Hence, we have $\mathcal{R}f \in \mathcal{B}^\alpha_{p,\infty}$ in this case.\smallskip
  
  \textit{Step~2}: We will establish the reconstruction bound~\eqref{eq:resonstruction bound} for $\mathcal{R}f$. For fixed $x \in \R^d$, $\lambda \in (0,1]$ and $\eta \in \mathcal{B}^r$, we have
  \begin{equation*}
    \langle \mathcal{R}f - \Pi_xf(x), \eta^\lambda_x\rangle = \sum_{\psi \in \Psi}\sum_{n \geq 0} \sum_{y \in \Lambda_n}\langle \mathcal{R}f - \Pi_xf(x), \psi^n_y\rangle\langle\psi^n_y,\eta^\lambda_x\rangle + \sum_{y \in \Lambda_0}\langle \mathcal{R}f - \Pi_xf(x), \varphi_y\rangle \langle \varphi_y, \eta^\lambda_x\rangle,
  \end{equation*}
  where 
  \begin{align*}
    \langle \mathcal{R}f - \Pi_xf(x), \psi^n_y\rangle &= \langle \Pi_y\overline{f}^n(y) - \Pi_xf(x), \psi^n_y \rangle \\
    &= \int_{B(y,2^{-n})} 2^{n|\s|}\langle \Pi_y(\Gamma_{y,z}f(z) - \Gamma_{y,x}f(x)), \psi^n_y \rangle \dd z,
  \end{align*}
  and the same expression holds for $\langle \mathcal{R}f - \Pi_xf(x), \varphi_y\rangle$. It follows that
  \begin{align}\label{eq:estimate of the coefficient of one model}
    \Big| \langle \mathcal{R}f - \Pi_x &f(x), \psi^n_y\rangle \Big| 
    \leq \int_{B(y,2^{-n})} 2^{n|\s|} \Big|\langle \Pi_y(\Gamma_{y,z}f(z) - \Gamma_{y,x}f(x)), \psi^{2^{-n}}_y\rangle\Big|  \dd z \, 2^{-n\frac{|\s|}{2}} \nonumber\\
    &\lesssim \|\Pi\| \sum_{\zeta \in A_\gamma} \int_{B(y,2^{-n})} 2^{n|\s|}  |\Gamma_{y,z}f(z) - \Gamma_{y,x}f(x)|_\zeta \dd z\, 2^{-n\zeta-n\frac{|\s|}{2}}\nonumber \\
    &\lesssim \|\Pi\|\|\Gamma\|\sum_{\zeta \in A_\gamma}\sum_{\gamma >\beta \geq \zeta}\int_{B(y,2^{-n})} 2^{n|\s|}  |f(z) - \Gamma_{z,x}f(x)|_{\beta} \|z - y \|_{\s}^{\beta - \zeta} \dd z\, 2^{-n\zeta-n\frac{|\s|}{2}} \nonumber\\
    &\lesssim \sum_{\zeta \in A_\gamma}\sum_{\gamma >\beta \geq \zeta}\int_{B(y,2^{-n})} 2^{n|\s|}  |f(z) - \Gamma_{z,x}f(x)|_\beta \dd z \,2^{-n\beta-n\frac{|\s|}{2}},
  \end{align}
  uniformly over all $x \in \R^d$, $n \geq 0$ and $y \in \Lambda_n$.
  
  As in the proof of Theorem~3.1 in~\cite{Hairer2017}, for a given $\lambda \in (0,1]$, there exists a $n_0 \geq 0$ such that $\lambda \in (2^{-n_0-1},2^{-n_0}]$. Let $\| \cdot \|_{L^q_{n_0}(\d\lambda)}$ denote the $L^q$-norm with respect to the finite measure (with the total mass $\ln2$) $\lambda^{-1}\1_{(2^{-n_0-1},2^{-n_0}]}\dd\lambda$, we first bound the term
  \begin{equation}\label{eq:term 1}
    \left\lVert \Big\| \sup_{\eta \in \mathcal{B}^r}\frac{|\sum_{n \le n_0}\sum_{y \in \Lambda_n}\langle \mathcal{R}f - \Pi_xf(x), \psi^n_y \rangle \langle \psi^n_y, \eta^\lambda_x \rangle |}{\lambda^\gamma} \Big\|_{L^p(\d x)} \right\rVert_{L^q_{n_0}(\d \lambda)}.
  \end{equation}
  Since for $n \leq n_0$ one has $\lambda  \le 2^{-n}$ and consequently
  $$
    |\langle\psi^n_y, \eta^\lambda_x\rangle| \lesssim 2^{n\frac{|\s|}{2}},
  $$
  uniformly over all $y \in \Lambda_n$, all $n \le n_0$, all $\eta \in \mathcal{B}^r$ and all $x \in \R^d$. Moreover, this inner product vanishes as soon as $\|x - y\|_{\s} > C2^{-n}$ for some constant $C > 0$ only depending on the size of the support of $\psi$. Hence, using all estimates above we get
  \begin{align*}
    &\left\lVert \Big\| \sup_{\eta \in \mathcal{B}^r}\frac{|\sum_{n \le n_0}\sum_{y \in \Lambda_n}\langle \mathcal{R}f - \Pi_x f(x), \psi^n_y \rangle \langle \psi^n_y, \eta^\lambda_x \rangle |}{\lambda^\gamma} \Big\|_{L^p(\d x)} \right\rVert_{L^q_{n_0}(\d \lambda)} \\
    &\quad\lesssim \sum_{n \le n_0} \sum_{\beta \in A_\gamma}\Big\|\sum_{y \in \Lambda_n, \|y - x\|_{\s} \le C2^{-n}}\int_{B(y,2^{-n})} 2^{n|\s|}\frac{|f(z) - \Gamma_{z,x}f(x)|_\beta}{2^{-n_0\gamma}} \dd z\, 2^{-n\beta}\Big\|_{L^p(\d x)} \\
    &\quad\lesssim \sum_{n \le n_0} \sum_{\beta \in A_\gamma}2^{(n_0 - n)\gamma}\Big\|\int_{B(0,C^\prime2^{-n})} 2^{n|\s|}\frac{|f(x + h) - \Gamma_{x + h ,x}f(x)|_\beta}{2^{-n(\gamma - \beta)}} \dd h \Big\|_{L^p(\d x)} \\
    &\quad\lesssim \sum_{\beta \in A_\gamma}\sum_{n \le n_0}2^{(n_0 - n)\gamma}\int_{h \in B(0,C^\prime2^{-n})}2^{n|\s|}\Big\|\frac{|f(x + h) - \Gamma_{x + h ,x}f(x)|_\beta}{\Vert h \Vert_{\s}^{\gamma - \beta}}\Big\|_{L^p(\d x)} \dd h,
  \end{align*}
  where we used 
  Minkowski's integral inequality.
  Since $\gamma < 0$, we have $\sum_{n \le n_0} 2^{(n_0- n)\gamma} \lesssim 1$ uniformly over all $n_0 \geq 0$, (Note, that this step constitute the main difference compared to the construction of the reconstruction operator with positive regularity, cf. \cite{Hairer2017}.) Therefore, we can apply Jensen's inequality for finite discrete measures $n \in \{0,\dots, n_0\} \mapsto 2^{(n_0-n)\gamma}$ to obtain
  \begin{align}  
    &\left \lVert \left\lVert \Big\| \sup_{\eta \in \mathcal{B}^r}\frac{|\sum_{n \le n_0}\sum_{y \in \Lambda_n}\langle \mathcal{R}f - \Pi_xf(x), \psi^n_y \rangle \langle \psi^n_y, \eta^\lambda_x \rangle |}{\lambda^\gamma} \Big\|_{L^p(\d x)} \right\rVert_{L^q_{n_0}(\d \lambda)} \right \rVert_{\ell^q(n_0 \geq 0)}\nonumber \\ 
    &\quad\lesssim \sum_{\beta \in A_\gamma}\bigg(\sum_{ n_0 \ge 0}\bigg(\sum_{n \le n_0}2^{(n_0-n)\gamma}\int_{h \in B(0,C^\prime2^{-n})}2^{n|\s|}\Big\|\frac{|f(x + h) - \Gamma_{x + h ,x}f(x)|_\beta}{\Vert h \Vert_{\s}^{\gamma - \beta}}\Big\|_{L^p(\d x)} \dd h  \bigg)^q\bigg)^{\frac{1}{q}}\nonumber\\
    &\quad\lesssim  \sum_{\beta \in A_\gamma}\bigg(\sum_{ n_0 \ge 0}\sum_{n \le n_0}2^{(n_0-n)\gamma}\bigg(\int_{h \in B(0,C^\prime2^{-n})}2^{n|\s|}\Big\|\frac{|f(x + h) - \Gamma_{x + h ,x}f(x)|_\beta}{\Vert h \Vert_{\s}^{\gamma - \beta}}\Big\|_{L^p(\d x)} \dd h\bigg)^q  \bigg)^{\frac{1}{q}} \nonumber\\
    &\quad\lesssim \sum_{\beta \in A_\gamma} \bigg(\sum_{n \ge 0}\sum_{ n_0 \ge n}2^{(n_0-n)\gamma}\int_{h \in B(0,C^\prime 2^{-n})}2^{n|\s|}\Big\|\frac{|f(x + h) - \Gamma_{x + h ,x}f(x)|_\beta}{\Vert h \Vert_{\s}^{\gamma - \beta}}\Big\|_{L^p(\d x)}^q \,\d h \bigg)^{\frac{1}{q}} \nonumber\\
    &\quad\lesssim \sum_{\beta \in A_\gamma} \bigg(\int_{h \in B(0,C^\prime)}\Big\|\frac{|f(x + h) - \Gamma_{x + h ,x}f(x)|_\beta}{\Vert h \Vert_{\s}^{\gamma - \beta}}\Big\|_{L^p(\d x)}^q\, \frac{\d h}{\Vert h\Vert_{\s}^{|\s|}} \bigg)^{\frac{1}{q}}\nonumber \\
    &\quad\lesssim (1 + \Vert \Gamma \Vert) \vertiii{f}_{\gamma,p,q}\label{eq:page13},
  \end{align}
  where we used Jensen's inequality in the third line for the finite measure $2^{n|\s|}\dd h|_{B(0,C^\prime 2^{-n})}$ and the definition of $\vertiii{f}_{\gamma,p,q}$ for $f \in \mathcal{D}^{\gamma}_{p,q}$ in the last line; and from the fourth line to the fifth line we implicitly applied repeated decompositions of $B(0,C^\prime)$ into the disjoint union of annuli $2^{-n-k-1} < \Vert h \Vert_{\s} \le 2^{-n-k}$ and the obvious relation that $2^{n|\s|} \le 2^{-k|\s|}1/\Vert h \Vert_{\s}^{|\s|}$ for all $h$ with $2^{-n-k-1} < \Vert h \Vert_{\s} \le 2^{-n-k}$. Obviously, the same bound holds if we replace $\sum_{n \le n_0}\sum_{y \in \Lambda_n}\langle \mathcal{R}f - \Pi_xf(x), \psi^n_y \rangle \langle \psi^n_y, \eta^\lambda_x \rangle$ by 
  $$
    \sum_{n \le n_0}\sum_{y \in \Lambda_n}\langle \mathcal{R}f - \Pi_xf(x), \psi^n_y \rangle \langle \psi^n_y, \eta^\lambda_x \rangle +\sum_{y \in \Lambda_0}\langle \mathcal{R}f - \Pi_xf(x), \varphi_y\rangle \langle \varphi_y, \eta^\lambda_x\rangle.
  $$ 
  
  Now we turn to the term
  \begin{equation}\label{eq:term 2}
    \left\lVert \Big\| \sup_{\eta \in \mathcal{B}^r}\frac{|\sum_{n > n_0}\sum_{y \in \Lambda_n}\langle \mathcal{R}f - \Pi_xf(x), \psi^n_y \rangle \langle \psi^n_y, \eta^\lambda_x \rangle |}{\lambda^\gamma} \Big\|_{L^p(\d x)} \right\rVert_{L^q_{n_0}(\d \lambda)}.
  \end{equation}
  In this case, we have (cf. the proof of \cite[Theorem~3.10]{Hairer2014})
  $$
    |\langle \psi^n_y, \eta^\lambda_x\rangle| \lesssim 2^{-n\frac{|\s|}{2}-rn}\lambda^{-|\s|-r}
  $$
  uniformly over all $n > n_0$, all $\lambda \in(2^{-n_0-1},2^{-n_0}]$ and all $\eta \in \mathcal{B}^r$. Moreover, this inner product vanishes as soon as $\|y - x\|_{\s}> C\lambda$ for some constant $C$ only depending on the size of the support of $\psi$. Hence, as we have shown that
  $$
    \Big| \langle \mathcal{R}f - \Pi_xf(x), \psi^n_y\rangle \Big| \lesssim \sum_{\zeta \in A_\gamma}\sum_{\gamma >\beta \geq \zeta}\int_{B(y,2^{-n})} 2^{n|\s|}  |f(z) - \Gamma_{z,x}f(x)|_\beta \dd z \,2^{-n\beta-n\frac{|\s|}{2}},
  $$
  it yields that
  \begin{align*}
    & \left\lVert \Big\| \sup_{\eta \in \mathcal{B}^r}\frac{|\sum_{n > n_0}\sum_{y \in \Lambda_n}\langle \mathcal{R}f - \Pi_xf(x), \psi^n_y \rangle \langle \psi^n_y, \eta^\lambda_x \rangle |}{\lambda^\gamma} \Big\|_{L^p(\d x)} \right\rVert_{L^q_{n_0}(\d \lambda)}\\
    &\quad \lesssim \sum_{\beta \in A_\gamma} \sum_{n > n_0}\Big\|\sum_{y \in \Lambda_n, \|y - x\|_{\s} \leq C2^{-n_0}}\int_{B(y,2^{-n})}2^{n|\s|} \frac{|f(z) - \Gamma_{z,x}f(x)|_\beta}{2^{-n_0(\gamma + |\s| + r)}} \dd z\, 2^{-n(|\s|+ \beta + r)}\Big\|_{L^p(\d x)} \\
    &\quad\lesssim \sum_{\beta \in A_\gamma} \sum_{n > n_0} \Big\|\int_{h \in B(0,C^\prime2^{-n_0})}2^{n|\s|} \frac{|f(x+h) - \Gamma_{x+h,x}f(x)|_\beta}{2^{-n_0(\gamma + |\s| + r)}} \dd h\, 2^{-n(|\s|+ \beta + r)}\Big\|_{L^p(\d x)} \\
    &\quad\lesssim \sum_{\beta \in A_\gamma} \sum_{n > n_0}\Big\|\int_{h \in B(0,C^\prime2^{-n_0})}2^{n_0|\s|} \frac{|f(x+h) - \Gamma_{x+h,x}f(x)|_\beta}{\|h\|_{\s}^{\gamma - \beta}} \dd h \Big\|_{L^p(\d x)}2^{-(n-n_0)(\beta + r)}, 
  \end{align*}
  where we implicitly used the relation that 
  $$
    \sum_{y \in \Lambda_n, \|y - x \|_{\s} \leq C2^{-n_0}}\1_{B(y,2^{-n})} \lesssim \1_{B(x, C^\prime 2^{-n_0})}
  $$ 
  holds uniformly for all $n > n_0$. Then, using Jensen's inequality for the finite measure $2^{n|\s|}\dd h|_{B(0,C^\prime 2^{-n})}$ and Minkowski's integral inequality for the $\| \cdot \|_{L^p}$-norm and the integral with respect to $h$, we can further deduce that
  \begin{align*}
    &\left\lVert \Big\| \sup_{\eta \in \mathcal{B}^r}\frac{|\sum_{n > n_0}\sum_{y \in \Lambda_n}\langle \mathcal{R}f - \Pi_xf(x), \psi^n_y \rangle \langle \psi^n_y, \eta^\lambda_x \rangle |}{\lambda^\gamma} \Big\|_{L^p(\d x)} \right\rVert_{L^q_{n_0}(\d \lambda)} \\
    &\quad \lesssim \sum_{\beta \in A_\gamma} \sum_{n > n_0}\int_{h \in B(0,C^\prime2^{-n_0})}2^{n_0|\s|} \Big\|\frac{|f(x+h) - \Gamma_{x+h,x}f(x)|_\beta}{\|h\|_{\s}^{\gamma - \beta}} \Big\|_{L^p(\d x)}\dd h\, 2^{-(n-n_0)(\beta + r)}.
  \end{align*}
  Since $r \in \N$ satisfies that $r > |\alpha|$ and since $\gamma < 0$ such that $\alpha \leq \beta <0$ holds for all $\beta \in \mathcal{A}_\gamma$, one has $\beta + r > 0$ for all $\beta \in \mathcal{A}_\gamma$. This implies that we can apply Jensen's inequality for the discrete finite measures $n \in \{n_0+1,\dots\}\mapsto  2^{-(n-n_0)(\beta + r)}$ and obtain that 
  \begin{align*}
    &\left \lVert\left\lVert \Big\| \sup_{\eta \in \mathcal{B}^r}\frac{|\sum_{n > n_0}\sum_{y \in \Lambda_n}\langle \mathcal{R}f - \Pi_xf(x), \psi^n_y \rangle \langle \psi^n_y, \eta^\lambda_x \rangle |}{\lambda^\gamma} \Big\|_{L^p(\d x)} \right\rVert_{L^q_{n_0}(\d \lambda)}\right \rVert_{\ell^q(n_0\ge 0)}\\
    &\quad\lesssim\sum_{\beta \in A_\gamma}\bigg( \sum_{ n_0 \ge 0} \sum_{n > n_0} 2^{-(n-n_0)(\beta + r)} \\
    &\hspace{3.5cm}\times \bigg(\int_{h \in B(0,C^\prime2^{-n_0})}2^{n_0|\s|} \Big\|\frac{|f(x+h) - \Gamma_{x+h,x}f(x)|_\beta}{\|h\|_{\s}^{\gamma - \beta}} \Big\|_{L^p(\d x)}\dd h\bigg)^q \bigg)^{\frac{1}{q}} \\
    &\quad\lesssim \sum_{\beta \in A_\gamma}\bigg(\int_{h \in B(0,C^\prime)} \Big\|\frac{|f(x+h) - \Gamma_{x+h,x}f(x)|_\beta}{\|h\|_{\s}^{\gamma - \beta}} \Big\|_{L^p(\d x)}^q\,\frac{\d h}{\|h\|_{\s}^{|\s|}}\bigg)^{\frac{1}{q}} \\
    &\quad\lesssim (1 + \|\Gamma\|)\vertiii{f}_{\gamma,p,q},
  \end{align*}
  where we used Jensen's inequality in the third line, and used the same reasoning as in the inequality~\eqref{eq:page13} (from the fourth line to the fifth line) in the last step.
  
  Now we note that for any measurable function $g$ defined on $(0,1]$, it holds that $\|g\|_{L^q_\lambda} = \Big\|\|g\|_{L^q_{n_0}}\Big\|_{\ell^q(n_0 \ge 0)}$. This implies that  $\left\lVert \Big\| \sup_{\eta \in \mathcal{B}^r} \frac{|\langle \mathcal{R}f - \Pi_xf(x), \eta^\lambda_x \rangle|}{\lambda^\gamma} \Big\|_{L^p} \right\rVert_{L^q_\lambda}$ is equal to the sum of the $\ell^q$-norms of~\eqref{eq:term 1} and~\eqref{eq:term 2}, therefore it is bounded by $\vertiii{f}_{\gamma,p,q}$ due to the above estimates. This proves~\eqref{eq:resonstruction bound}. \smallskip
  
  \textit{Step 3:} Now suppose that $(\overline{\Pi}, \overline{\Gamma})$ is another model and we use $\overline{\mathcal{D}}^\gamma_{p,q}$ to denote the corresponding space of modelled distributions in the sense of Definition~\ref{def:Besov modelled distribution}. Fix a wavelet analysis $\{\varphi, \psi \in \Psi\}$ in $\mathcal{C}^r_0$, for given $f \in \mathcal{D}^\gamma_{p,q}$ and $g \in \overline{\mathcal{D}}^\gamma_{p,q}$ we define $\mathcal{R}f$ and $\overline{\mathcal{R}}g$ as in~\eqref{eq:reconstruction operator}:
  $$
    \mathcal{R} f := \sum_{n\in \N} \sum_{x\in \Lambda_n} \sum_{\psi \in \Psi}  \langle \Pi_x \overline{f}^n (x), \psi^n_x \rangle   \psi^n_x + \sum_{x \in \Lambda_0}\langle \Pi_x \overline{f}^0 (x), \varphi^0_x   \rangle   \varphi^0_x,
  $$
  $$
    \overline{\mathcal{R}} g := \sum_{n\in \N} \sum_{x\in \Lambda_n} \sum_{\psi \in \Psi}  \langle \overline{\Pi}_x \overline{g}^n (x), \psi^n_x \rangle   \psi^n_x + \sum_{x \in \Lambda_0}\langle \overline{\Pi}_x \overline{g}^0 (x), \varphi^0_x   \rangle   \varphi^0_x.
  $$
  From the results obtained in Step~1 and Step~2 we see that $\mathcal{R}f, \overline{\mathcal{R}}g$ are elements in $\mathcal{B}^{\bar{\alpha}}_{p,q}$ and satisfy the bound~\eqref{eq:resonstruction bound} for $(\Pi,\Gamma)$ and $(\overline{\Pi},\overline{\Gamma})$, respectively. 
  
  Now, for every $n \geq 0$, $x \in \R^d$, $y \in \Lambda_n$ and $\psi \in \Psi$, we can check that
  \begin{align}\label{eq:coefficient of the difference of two models}
  \begin{split}
    \langle \mathcal{R}f - & \overline{\mathcal{R}}g- \Pi_xf(x) + \overline{\Pi}_xg(x), \psi^n_y \rangle \\
    &= \langle \Pi_y\overline{f}^n(y) - \overline{\Pi}_y\overline{g}^n(y) - \Pi_xf(x) + \overline{\Pi}_xg(x), \psi^n_y \rangle \\
    &= \int_{z \in B(y,2^{-n})}2^{n|\s|}\langle \Pi_y\Gamma_{y,z}(f(z) - \Gamma_{z,x}f(x)) - \overline{\Gamma}_y\overline{\Gamma}_{y,z}(g(z) - \overline{\Gamma}_{z,x}g(x)), \psi^n_y \rangle \dd z.
  \end{split}
  \end{align}
  Since 
  \begin{align*}
    \Pi_y\Gamma_{y,z}(& f(z) - \Gamma_{z,x} f(x)) - \overline{\Gamma}_y\overline{\Gamma}_{y,z}(g(z) - \overline{\Gamma}_{z,x}g(x)) \\
    &= \Pi_y\Gamma_{y,z}(f(z) - \Gamma_{z,x}f(x) - g(z) 
    + \overline{\Gamma}_{z,x}g(x)) + (\Pi_y - \overline{\Pi}_y)\Gamma_{y,z}(g(z) - \overline{\Gamma}_{z,x}g(x))\\
    &\quad + \overline{\Pi}_y(\Gamma_{y,z} - \overline{\Gamma}_{y,z})(g(z) - \overline{\Gamma}_{z,x}g(x)),
  \end{align*} 
  we can get the following bound for~\eqref{eq:coefficient of the difference of two models}:
  \begin{align*}
    |\langle \mathcal{R} & f - \overline{\mathcal{R}}g- \Pi_xf(x) + \overline{\Pi}_xg(x), \psi^n_y \rangle|\\ 
    \leq & \|\Pi\|\|\Gamma\|\sum_{\zeta \in A_\gamma}\sum_{\gamma >\beta \ge \zeta}\int_{z \in B(y,2^{-n})}2^{n|\s|}|f(z) - \Gamma_{z,x}f(x) - g(z) + \overline{\Gamma}_{z,x}g(x)|_{\beta}\dd z\, 2^{-n\beta - \frac{n|\s|}{2}} \\
    &+ \|\Pi - \overline{\Pi}\|\|\Gamma\|\sum_{\zeta \in A_\gamma}\sum_{\gamma >\beta \ge \zeta}\int_{z \in B(y,2^{-n})}2^{n|\s|}|g(z) - \overline{\Gamma}_{z,x}g(x)|_{\beta}\dd z \, 2^{-n\beta - \frac{n|\s|}{2}} \\
    &+ \|\overline{\Pi}\|\|\Gamma - \overline{\Gamma}\|\sum_{\zeta \in A_\gamma}\sum_{\gamma >\beta \ge \zeta}\int_{z \in B(y,2^{-n})}2^{n|\s|}|g(z) - \overline{\Gamma}_{z,x}g(x)|_{\beta}\dd z\,2^{-n\beta - \frac{n|\s|}{2}}.
  \end{align*}
  Hence, by replacing the integrand $|f(z) - \Gamma_{z,x}f(x)|_{\beta}$ by $|f(z) - \Gamma_{z,x}f(x) - g(z) + \overline{\Gamma}_{z,x}g(x)|_{\beta}$ and $|g(z) - \overline{\Gamma}_{z,x}g(x)|_{\beta}$ in the estimate~\eqref{eq:estimate of the coefficient of one model}, we can apply the same arguments for establishing~\eqref{eq:resonstruction bound} in the Step~2 to obtain the bound~\eqref{eq:reconstruction bound for two models}, and complete the proof.
\end{proof}

\begin{remark}
  While we equip the spaces of modelled distributions with Besov norms, we kept the original definition of models, which comes with H\"older type estimates. This seems to be the reason why the reconstruction operator in Theorem~\ref{thm:reconstruction negative} maps, in general, modelled distributions to generalized functions with a slightly lower Besov regularity, cf. Remark~3.2 in \cite{Hairer2017}. However, to generalize the definition of models to models with Besov type estimates is outside the scope of the present article and left for future research since the H\"older type estimates allow for various Besov type generalizations and the natural choice might depend on the specific application in mind.  
\end{remark}

\subsection{Reconstruction theorem for models with local bounds}\label{subsec:reconstruction for local models}

The original definition of models (recall Definition~\ref{def:local model}) requires the bounds in~\eqref{eq:local model} to hold only \textit{locally}, that means to hold on every compact set. In the light of stochastic integration it seems to be natural to assume global bounds in the definition of models (recall Definition~\ref{def:global model}), see Section~\ref{sec:stochastic integration}. However, using the same arguments as in the proof of Theorem~\ref{thm:reconstruction negative} or of Theorem~3.1 in~\cite{Hairer2017}, one can obtain the reconstruction theorem for models with local bounds and, consequently, for local Besov spaces of modelled distributions. The only difference is to carry out the arguments on every compact set~$\mathcal{K}\subset \R^d$ instead of $\R^d$.

For the rest of this subsection we fix a regularity structures $\mathcal{T}=(A,T,G)$ with a model~$(\Pi,\Gamma)$ in the sense of Definition~\ref{def:local model}. The local Besov spaces are defined in the obvious manner.

\begin{definition}
  Let $\gamma\in \R$ and $p,q\in [1,\infty)$.  Let $\alpha \in \R $ and $r\in \N$ be such that $r>|\alpha|$. 
  \begin{itemize}
    \item  The \textit{local Besov space}~$\mathcal{D}_{p,q}^{\gamma, \textup{loc}}$ of modelled distributions consists of all functions $f\colon\R^{d}\to T_{\gamma}^{-}$ such that $ \interleave f\interleave_{\gamma,p,q,\mathcal{K}}<\infty$ for every compact set $\mathcal{K}\subset \R^d$.
    \item For $\alpha < 0$ the \textit{local Besov space}~$\mathcal{B}^{\alpha,\textup{loc}}_{p,q}$ is the space of all distributions~$\xi$ on $\R^d$ such that, for every compact set $\mathcal{K}\subset \R^d$,
    \begin{equation*}
      \bigg\| \Big\| \sup_{\eta\in \mathcal{B}^r} \frac{|\langle \xi , \eta^\lambda_x \rangle|}{\lambda^\alpha}  \Big\|_{L^p(\mathcal{K})} \bigg\|_{L^q_\lambda} <\infty. 
    \end{equation*}
    \item  For $\alpha \geq 0$ the \textit{local Besov space}~$\mathcal{B}^{\alpha,\textup{loc}}_{p,q}$ is the space of all distributions~$\xi$ on $\R^d$ such that, for every compact set $\mathcal{K}\subset \R^d$,
    \begin{equation*}
      \bigg\| \sup_{\eta\in \mathcal{B}^r} |\langle \xi , \eta_x^1 \rangle|   \bigg\|_{L^p(\mathcal{K})} <\infty
      \quad \text{and}\quad
      \bigg\| \Big\| \sup_{\eta\in \mathcal{B}^r_{\lfloor \alpha\rfloor}} \frac{|\langle \xi , \eta^\lambda_x \rangle|}{\lambda^\alpha}  \Big\|_{L^p(\mathcal{K})} \bigg\|_{L^q_\lambda} <\infty. 
    \end{equation*}
  \end{itemize}
\end{definition}

Based on these local versions of Besov spaces, the reconstruction theorem for modelled distributions with negative regularity reads as follows.

\begin{corollary}\label{cor:reconstruction theorem with local models}
  Let $\mathcal{T}=(A,T,G)$ be a regularity structures with a model~$(\Pi,\Gamma)$ in the sense of Definition~\ref{def:local model}.
  \begin{enumerate}
    \item Suppose that $\alpha := \min A < \gamma < 0 $. If $q = \infty $, let $\bar{\alpha} = \alpha$, otherwise take $\bar{\alpha} < \alpha$. Then, there exists a continuous linear operator $\mathcal{R} \colon \mathcal{D}^{\gamma,\textup{loc}}_{p,q} \rightarrow \mathcal{B}^{\bar{\alpha},\textup{loc}}_{p,q}$ such that 
    \begin{equation}\label{eq:local reconstruction}
      \left\lVert \Big\| \sup_{\eta \in \mathcal{B}^r} \frac{|\langle \mathcal{R}f - \Pi_xf(x), \eta^\lambda_x \rangle|}{\lambda^\gamma} \Big\|_{L^p(\mathcal{K})} \right\rVert_{L^q_\lambda} \lesssim \|\Pi\|_{\gamma,\mathcal{K}}(1 + \|\Gamma\|_{\gamma,\mathcal{K}})\vertiii{f}_{\gamma,p,q,\mathcal{K}},
    \end{equation}
    holds uniformly over all $f \in \mathcal{D}^{\gamma,\textup{loc}}_{p,q}$ and for every compact set $\mathcal{K}\subset \R^d$.
    \item Suppose that $\gamma >0$ and $\alpha := \min(A\setminus \N)\wedge \gamma $. If $q = \infty $, let $\bar{\alpha} = \alpha$, otherwise take $\bar{\alpha} < \alpha$. Then, there exists a unique continuous linear operator $\mathcal{R} \colon \mathcal{D}^{\gamma,\textup{loc}}_{p,q} \rightarrow \mathcal{B}^{\bar{\alpha},\textup{loc}}_{p,q}$ such that~\eqref{eq:local reconstruction} holds uniformly over all $f \in \mathcal{D}^{\gamma,\textup{loc}}_{p,q}$ and for every compact set $\mathcal{K}\subset \R^d$.
  \end{enumerate}

  Furthermore, let $\overline{\mathcal{R}}$ be a reconstruction operator w.r.t. to another model $(\overline{\Pi},\overline{\Gamma})$ in the sense of Definition~\ref{def:local model} for $\mathcal{T}=(A,T,G)$. Then, in both above cases, $\overline{\mathcal{R}}$ and $\mathcal{R}$ satisfy the localized version of~\eqref{eq:reconstruction bound for two models} for every compact set $\mathcal{K}\subset \R^d$.   
\end{corollary}

\begin{proof}
  1. The reconstruction operator~$\mathcal{R}$ is constructed as before, see~\eqref{eq:reconstruction operator}. In order obtain the reconstruction theorem for local Besov spaces of modelled distribution, the only change in the proof of the convergence, of the bound~\eqref{eq:local reconstruction} and of continuity with respect to the models (i.e.~\eqref{eq:reconstruction bound for two models}) is to replace the norm $\|\cdot \|_{L^p}$ with $\|\cdot \|_{L^p(\mathcal{K})}$ and carry out exactly the same arguments for every compact set~$\mathcal{K}\subset \R^d$. 

  2. In case of positive regularity~$\gamma$, the reconstruction operator is defined as on page~2603 in~\cite{Hairer2017} and again the estimates and arguments as given in the proof of Theorem~3.1 in \cite{Hairer2017} transfer line by line to the setting of local Besov spaces. This observation was already made in the case of Sobolev--Slobodeckij in an early version of the present article and also pointed out for the more general case of Besov spaces in~\cite{Hairer2017}, see Remark~2.9 in~\cite{Hairer2017}. Secondly, let us remark that the assumption that the polynomial regularity is included in the considered regularity structure $\mathcal{T}=(A,T,G)$ is not necessary, see also page~2596 in~\cite{Hairer2017}.
\end{proof}

\section{Stochastic integration on spaces of modelled distributions}\label{sec:stochastic integration}

This section is devoted to prove that the Besov spaces~$\mathcal{D}_{p,q}^{\gamma}$ of modelled distributions are UMD Banach spaces and of martingale type~$2$. These Banach space properties open the door to apply highly developed stochastic integration theory on the spaces~$\mathcal{D}_{p,q}^{\gamma}$. For example one can integrate predictable $\mathcal{D}_{p,q}^{\gamma}$-valued processes with respect to Brownian motion. One successful application of stochastic integration on Banach spaces lies in the area of stochastic partial differential equations, see e.g. \cite{Brzezniak1995} and \cite{vanNeerven2007}, which we will discuss in more details in Section~\ref{sec:SPDEs} below. For a more comprehensive introduction and treatment of stochastic integration on Banach spaces we refer for instance to \cite{Dalang2015,Mandrekar2015}.\smallskip

Let $(\Omega,\mathcal{F},\mathbb{F},\mathbb{P})$ be a complete filtered probability space, $I\subset\mathbb{R}$, $\mathbb{F}:=(\mathcal{F}_{t})_{t\in I}$ be an increasing family of sub-$\sigma$-algebra of $\mathcal{F}$ and $X$ be a Banach space with norm $\|\cdot\|_{X}$. The expectation operator with respect to $\mathbb{P}$ is denoted~$\mathbb{E}$ and the corresponding conditional expectation by $\mathbb{E}[\,\cdot\,|\mathcal{F}_{t}]$ for $t\in I$. A process $(M_{t})_{t\in I}$ is a \textit{$X$-valued martingale} if and only if $M_{t}\in L^{1}(\Omega,\mathcal{F}_{t},\mathbb{P};X)$ for all $t\in I$ and 
\begin{equation*}
  \mathbb{E}[M_{t}|\mathcal{F}_{s}]=M_{s}\quad \P\text{-}a.s.,\quad \text{for all }s,t\in I\text{ with }s\leq t.
\end{equation*}
A sequence $(\xi_{i})_{i\in\mathbb{N}}$ is called \textit{martingale difference} if $(\sum_{i=0}^{n}\xi_{i})_{n\in\mathbb{N}}$ is a $X$-valued martingale. To rely on stochastic integration theory on Banach spaces, one needs to require some additional properties on the Banach space $X$. The definitions are taken from~\cite{Brzezniak1995}, see Definition~2.1 and Definition~B.2 therein.

\begin{definition}
  Let $(\Omega,\mathcal{F},\mathbb{P})$ be a complete probability space.
  \begin{itemize}
    \item A Banach space $(X,\|\cdot \|_X)$ is of \textit{martingale type $p$} for $p\in[1,\infty)$ if any $X$-valued martingale $(M_n)_{n\in\mathbb{N}}$ satisfies 
          \begin{equation*}
            \sup_{n}\mathbb{E}[\|M_{n}\|_{X}^{p}]\leq C_{p}(X)\sum_{n\in\mathbb{N}}\mathbb{E}[\|M_{n}-M_{n-1}\|_{X}^{p}]
          \end{equation*}
          for some constant $C_{p}(X)>0$ independent of the martingale $(M_{n})_{n\in\mathbb{N}}$ and $M_{-1}:=0$.
    \item A Banach space $(X,\|\cdot \|_X)$ if of\textit{ type $p$} for $p\in[1,2]$ if any finite sequence $\epsilon_{1},\dots,\epsilon_{n}\colon\Omega\to\{-1,1\}$ of symmetric and i.i.d. random variables
          and for any finite sequence $x_{1},\dots,x_{n}$ of elements of $X$ the inequality 
          \begin{equation*}
            \mathbb{E}\bigg[\bigg\|\sum_{i=1}^{n}\epsilon_{i}x_{i}\bigg\|_{X}^{p}\bigg]\leq K_{p}(X)\sum_{i=1}^{n}\|x_{i}\|_{X}^{p}
          \end{equation*}
          holds for some constant $K_{p}(X)>0$.
    \item A Banach space $(X,\|\cdot \|_X)$ is called an \textit{UMD space} or is said to have the \textit{unconditional martingale difference property} if for any $p\in(1,\infty)$, for any martingale difference
          $(\xi_{j})_{j\in\mathbb{N}}$ and for any sequence $(\epsilon_{i})_{i\in\mathbb{N}}\subset\{-1,1\}$ the inequality 
          \begin{equation*}    
            \mathbb{E}\bigg[\bigg\|\sum_{i=1}^{n}\epsilon_{i}\xi_{j}\bigg\|_{X}^{p}\bigg]\leq\tilde{K}_{p}(X)\mathbb{E}\bigg[\bigg\|\sum_{i=1}^{n}\xi_{i}\bigg\|_{X}^{p}\bigg]
          \end{equation*}
          holds for all $n\in\mathbb{N}$, where $\tilde{K}_{p}(X)>0$ is some constant.
  \end{itemize}
\end{definition}

Let us remark that Hilbert spaces and finite dimensional Banach spaces are always UMD spaces.\smallskip

Coming back to a regularity structure $\mathcal{T}=(A,T,G)$ with an associated model $(\Pi,\Gamma)$ and let us assume now additionally that each $T_{\alpha}$ is an UMD space for $\alpha \in A$. Under this assumption the space $T_{\gamma}^{-}=\bigoplus_{\alpha<\gamma}T_{\alpha}$ is again an UMD space (Theorem~4.5.2 in \cite{Amann1995}) since $A$ is locally finite and $T$ is a finite product of UMD spaces. 

\begin{proposition}\label{prop:UMD property}
  Let $\mathcal{T}=(A,T,G)$ be a regularity structure with a model $(\Pi,\Gamma)$ as in the Definition~\ref{def:global model}. Suppose that $\gamma \in \R$ and that the Banach space $T_\alpha$ is an UMD space for every $\alpha \in A$. Then, the space $\mathcal{D}_{p,q}^{\gamma}$ is an UMD spaces, too, for $1 < p < \infty$ and $1 < q < \infty$. If the Banach space $T_\gamma^-$ is additionally of type $2$, then $\mathcal{D}^\gamma_{p,q}$ is of martingale type $2$ for every $p\geq2$ and $q \geq2$.
\end{proposition}

\begin{proof}
  Since every $T_\alpha$ with $\alpha \in A_\gamma$ is an UMD space by assumption, by Theorem~4.5.2 in~\cite{Amann1995} every $L^p(\R^d;T_\alpha)$ is also an UMD space. Furthermore, let $\mu$ be the Borel measure on $\R$ defined by 
  $$
    \mu(h) := \frac{1}{\|h\|_{\mathfrak{s}}^{|\mathfrak{s}|}} \dd h,
  $$
  the corresponding $L^q$-space $L^q_\mu(B(0,1); L^p(\R^d; T_\alpha))$ is again an UMD space for every $\alpha \in A_\gamma$ due to Theorem~4.5.2 in~\cite{Amann1995}. Consequently the finite product space 
  $$
    \prod_{\alpha \in A_\gamma} \Big(L^p(\R^d;T_\alpha) \times L^q_\mu(B(0,1); L^p(\R^d; T_\alpha)) \Big)
  $$
  is an UMD space. We will show that $\mathcal{D}^\gamma_{p,q}$ is a closed linear subspace in the above product space and then by Theorem~4.5.2 in~\cite{Amann1995} again we can conclude that $\mathcal{D}^\gamma_{p,q}$ is an UMD space. 
  
  For this purpose we define for every $\alpha \in A_\gamma$ the following mappings
  \begin{align*}
    \Phi^\alpha_1 \colon \mathcal{D}^\gamma_{p,q} \rightarrow L^p(\R^d;T_\alpha)\quad \text{via}\quad
    f \mapsto f^\alpha
  \end{align*}
  and 
  \begin{align*}
    \Phi^\alpha_2 \colon \mathcal{D}^\gamma_{p,q} \rightarrow L^q_\mu(B(0,1); L^p(\R^d; T_\alpha))\quad \text{via}\quad
    f \mapsto \Big[h \mapsto \frac{f^\alpha(\cdot +h) - (\Gamma_{\cdot + h ,\cdot}f(\cdot))^\alpha}{\|h\|_{\mathfrak{s}}^{\gamma - \alpha}}\Big],
  \end{align*}
  where $f^\alpha$ is the projection of $f$ onto $T_\alpha$ and $\frac{f^\alpha(\cdot +h) - (\Gamma_{\cdot + h ,\cdot}f(\cdot))^\alpha}{\|h\|_{\mathfrak{s}}^{\gamma - \alpha}}$ is an element in $L^p(\R^d; T_\alpha)$ such that 
  $$
    \frac{f^\alpha(\cdot +h) - (\Gamma_{\cdot + h ,\cdot}f(\cdot))^\alpha}{\|h\|_{\mathfrak{s}}^{\gamma - \alpha}}(x) = \frac{f^\alpha(x +h) - (\Gamma_{x + h ,x}f(x))^\alpha}{\|h\|_{\mathfrak{s}}^{\gamma - \alpha}}
  $$
  for all $x \in \R^d$.
  
  Clearly, the mapping $\Big(\Phi^\alpha_1 \times \Phi^\alpha_2\Big)_{\alpha \in A_\gamma}$ is an isometry from $\mathcal{D}^\gamma_{p,q}$ onto its image in the product space 
  $$
    \prod_{\alpha \in A_\gamma} \Big(L^p(\R^d;T_\alpha) \times L^q_\mu(B(0,1); L^p(\R^d; T_\alpha)) \Big),
  $$
  so that we can embed $ \mathcal{D}^\gamma_{p,q} $ into the above product space as a closed linear subspace. By Theorem~4.5.2 in \cite{Amann1995} the space $ \mathcal{D}^\gamma_{p,q}$ is therefore UMD, too. The previous construction is similar to Lemma~A.5 in \cite{Brzezniak1995}. Since every UMD space of type $2$ is a Banach space of martingale type $2$ as shown in Proposition~B.4 in \cite{Brzezniak1995}, one concludes that $L^{p}(\R^d;T_\alpha)$ and $L^q_\mu(B(0,1); L^p(\R^d;T_\alpha))$ are of martingale type $2$ for every $p\in[2,\infty)$, $q\in[2,\infty)$ and $\alpha \in A_\gamma$, and the same argument as before applies.
\end{proof}

We can now formulate and prove our main theorem. Like in the Fubini theorem the order of reconstruction and stochastic integration can be interchanged:

\begin{theorem}\label{thm:fubini theorem}
  Let $\gamma > \alpha_0 := \inf A$, $\alpha_0 \notin \mathbb{Z}$ and $\mathcal{T}=(A,T,G)$ be a regularity structure together with a model $(\Pi,\Gamma)$ as in Definition~\ref{def:global model} and $T_\alpha$ is an UMD space for every $\alpha \in A$. Let $(\Omega,\mathcal{F},\mathbb{F}, \P)$ be a complete filtered probability space and $W$ be Brownian motion on $[0,T]$ for some $T \in (0,\infty)$. Let $H$ be a $\mathcal{D}^\gamma_{p,q}$-valued process for some $1 <p <\infty$ and $1 < q < \infty$ which is locally $L^2$-stochastically integrable with respect to $W$, then the order of ``integration'' can be interchanged
  \begin{equation}\label{eq:Fubini type formula}
    \bigg \langle \mathcal{R} \big( {(H \bullet W)} \big) , \psi \bigg \rangle = \bigg( \big \langle \mathcal{R}(H),\psi \big \rangle \bullet W \bigg ) \, 
  \end{equation}
  for every test function $ \psi\in \mathcal{B}^r$ with $r > |\alpha_0|$. Here $(H \bullet W)$ stands for the stochastic integral of $H$ with respect to $W$ and $\mathcal{R}$ denotes a reconstruction operator for $\mathcal{T}=(A,T,G)$ and $(\Pi,\Gamma)$.
\end{theorem}

\begin{proof}
  \textit{Step~1:} First we assume that $H$ is an elementary process which can be written as
  $$
    H(\omega,t) = \sum_{n = 1}^N \sum_{m = 1}^M \1_{(t_{n-1},t_n]}(t) \1_{A_{mn}}(\omega) f_{mn}
  $$
  where $0 = t_0 < t_1 < \dots < t_N = T$, $A_{mn} \in \mathcal{F}_{t_{n-1}}$ for all $m = 1,\dots,M$ and are pairwise disjoint, $f_{mn} \in \mathcal{D}^\gamma_{p,q}$ for all $m$ and $n$. Here $\1_{A_{mn}}$ denotes the indicator function of the set~$A_{n,m}$.
  
  Then it holds that for all $t \in [0,T]$,
  $$
   (H \bullet W)_t = \sum_{n = 1}^N \sum_{m = 1}^M \1_{A_{mn}}(W_{t \wedge t_n} - W_{t \wedge t_{n-1}})f_{mn},
  $$
  and therefore $\mathcal{R}((H \bullet W)_t) =  \sum_{n = 1}^N \sum_{m = 1}^M \1_{A_{mn}}(W_{t \wedge t_n} - W_{t \wedge t_{n-1}})\mathcal{R}f_{mn}$ as well as
  $$
    \bigg \langle \mathcal{R} \big( {(H \bullet W)_t} \big) , \psi \bigg \rangle = \sum_{n = 1}^N \sum_{m = 1}^M \1_{A_{mn}}(W_{t \wedge t_n} - W_{t \wedge t_{n-1}})\langle \mathcal{R}f_{mn}, \psi \rangle.
  $$
  On the other hand, we have
  $$
    \big \langle \mathcal{R}(H)(\omega,t),\psi \big \rangle = \sum_{n = 1}^N\sum_{m = 1}^M \1_{(t_{n-1},t_n]}(t) \1_{A_{mn}}(\omega) \langle \mathcal{R}f_{mn}, \psi \rangle
  $$
  which is an real-valued elementary process. Hence, we indeed have
  $$
    \bigg( \big \langle \mathcal{R}(H),\psi \big \rangle \bullet W \bigg )_t = \sum_{n = 1}^N \sum_{m = 1}^M \1_{A_{mn}}(W_{t \wedge t_n} - W_{t \wedge t_{n-1}})\langle \mathcal{R}f_{mn}, \psi \rangle.
  $$
  Obviously now we obtain~\eqref{eq:Fubini type formula} for all elementary processes $H$. \smallskip
 
  \textit{Step~2:} Now suppose that $H$ is a $L^2$-stochastically integrable process. By Theorems~3.5 (It\^o isomorphism) and Theorem~3.6 in \cite{vanNeerven2007}, there exists a sequence $(H_n)_{n \ge 1}$ of elementary processes such that
  \begin{equation*}
    H_n \rightarrow H\quad \text{in} \quad L^2(\Omega, \mathbb{P}; \gamma(L^2([0,T],\d t);\mathcal{D}^\gamma_{p,q})),
  \end{equation*}
  where $\gamma(L^2([0,T],\d t);\mathcal{D}^\gamma_{p,q})$ denotes the space of $\gamma$-radonifying operators from the space $L^2([0,T],\d t)$ into $\mathcal{D}^\gamma_{p,q}$ (see Section~2.2 in \cite{vanNeerven2007}) and 
  $$
    (H \bullet W) = \lim  (H_n \bullet W)\quad \text{in}\quad L^2(\Omega; C([0,T]; \mathcal{D}^\gamma_{p,q})).
  $$
  Now we choose an $\bar{\alpha} < \alpha_0$ with $\lfloor \bar{\alpha} \rfloor = \lfloor \alpha_0 \rfloor$. By Theorem~3.1 in~\cite{Hairer2017} (for $\gamma > 0$) and Theorem~\ref{thm:reconstruction negative} (for $\gamma < 0$) we know that $\mathcal{R} \colon \mathcal{D}^{\gamma}_{p,q} \rightarrow \mathcal{B}^{\bar{\alpha}}_{p,q}$ is a continuous linear mapping, which implies that 
  $$
    \mathcal{R}  \big( {(H \bullet W)} \big) = \lim_{ n \rightarrow \infty} \mathcal{R}  \big( {(H_n \bullet W)} \big)
  $$
  uniformly in $t \in [0,T]$ with respect to the Besov topology on $\mathcal{B}^{\bar{\alpha}}_{p,q}$. Since $\mathcal{B}^{\bar{\alpha}}_{p,q}$ can be embedded in the dual of $\mathcal{C}^r_0$ for $r \geq  \lfloor \bar{\alpha} \rfloor = \lfloor \alpha_0 \rfloor$, we can derive that 
  $$
    \bigg \langle \mathcal{R} \big( {(H \bullet W)} \big) , \psi \bigg \rangle = \lim_{n \rightarrow \infty}  \bigg \langle \mathcal{R} \big( {(H_n \bullet W)} \big) , \psi \bigg \rangle
  $$
  in $L^2(\Omega; C([0,T];\R))$ for any $\psi \in \mathcal{B}^r \subset \mathcal{C}^r_0$. 
  
  On the other hand, since the operator $\mathcal{R}$ and the dual pairing $\langle \cdot, \psi \rangle$ are continuous, the ideal property of $\gamma$-radonifying operators (cf. Proposition~2.3 in~\cite{vanNeerven2007}) implies that $\langle \mathcal{R}(H), \psi \rangle$ is $L^2$-stochastically integrable with respect to $W$ and 
  $$
    \mathbb{E}\Big[\| \langle \mathcal{R}(H_n), \psi \rangle - \langle \mathcal{R}(H), \psi \rangle\|_{L^2([0,T],\d t)}^2\Big] \le \|\mathcal{R}\|^2\|\psi\|_{\mathcal{C}^r_0}^2\,\mathbb{E}\Big[\|H_n - H\|^2_{\gamma(L^2([0,T],\d t);\mathcal{D}^\gamma_{p,q})}\Big],
  $$ 
  which implies that $\langle \mathcal{R}(H_n), \psi \rangle
  $ converges to $\langle \mathcal{R}(H), \psi \rangle$ in $L^2(\Omega \times [0,T], \mathbb{P} \times \d t)$ as $n$ tends to infinity and therefore by It\^o isometry we obtain that 
  $$
    \bigg( \big \langle \mathcal{R}(H),\psi \big \rangle \bullet W \bigg ) = \lim_{n \rightarrow \infty} \bigg( \big \langle \mathcal{R}(H_n),\psi \big \rangle \bullet W \bigg ) 
  $$
  in $L^2(\Omega; C([0,T];\R))$ for any $\psi \in \mathcal{B}^r \subset \mathcal{C}^r_0$. Since we have 
  $$
    \bigg \langle \mathcal{R} \big( {(H_n \bullet W)} \big) , \psi \bigg \rangle = \bigg( \big \langle \mathcal{R}(H_n),\psi \big \rangle \bullet W \bigg )
  $$
  for every $n$ by the result from Step~1, we obtain~\eqref{eq:Fubini type formula} for such~$H$. \smallskip
 
  \textit{Step~3:} Now suppose that $H$ is locally $L^2$-stochastically integrable with respect to~$W$. A standard localization argument together with the result from Step~2 then provides that~\eqref{eq:Fubini type formula} holds for all such~$H$.
\end{proof}

\begin{remark}
  It is fairly straightforward to verify with the presented arguments that Proposition~\ref{prop:UMD property} and Theorem~\ref{thm:fubini theorem} still hold if one replaces the space~$\mathcal{D}^\gamma_{p,q}$ with~$\mathcal{D}^\gamma_{p,q}(\mathcal{K})$ for a compact set $\mathcal{K}\subset \R^d$. This allows to derive analogue versions of Proposition~\ref{prop:UMD property} and Theorem~\ref{thm:fubini theorem} for the local space~$\mathcal{D}^{\gamma,\textup{loc}}_{p,q}$. In particular, the space~$\mathcal{D}^{\gamma,\textup{loc}}_{p,q}$ is locally an UMD space and of martingale type $2$, which just means that the space $\mathcal{D}^\gamma_{p,q}$ appropriately factorized by functions~$(\psi_i)$ with vanishing norm satisfies the properties.
\end{remark}

\section{Semilinear SPDEs in spaces of modelled distributions}\label{sec:SPDEs}

A corner stone of the theory of regularity structures are the existence and uniqueness results for mild solutions of  semilinear (stochastic) partial differential equations in spaces of modelled distributions, see Section~$7$ and~$8$ in \cite{Hairer2014}. Hence, in order to get unique mild solutions for singular SPDEs like the KPZ equation or the stochastic quantisation equation, these equations are considered as semilinear (S)PDEs in the space of modelled distributions as opposed to classical function spaces.

Having shown that the Besov spaces of modelled distributions are UMD spaces and of $M$-type $2$, gives us access to the solution theories of SPDEs in these Banach spaces, see e.g. \cite{Brzezniak1995,Brzezniak1997,vanNeerven2008,Brzezniak2018}, and, consequently, we obtain novel existence and uniqueness results for mild solutions of semilinear SPDEs in spaces of modelled distributions. In the following we briefly illustrate this for SPDEs with finite dimensional noise but we would like to emphasize that the theory of SPDEs in Banach spaces works, of course, also in the case of infinite dimensional noises, cf. \cite{vanNeerven2008,Brzezniak2018}. 

\subsection{Existence and uniqueness of mild solutions}\label{subsec:existence of SPDE}

Throughout this section we assume: $\mathcal{T}=(A,T,G)$ is a regularity structure with an associated model $(\Pi,\Gamma) $ in the sense of Definition~\ref{def:global model}, $T_{\gamma}^{-}=\bigoplus_{\alpha<\gamma}T_{\alpha}$ is of $M$-type $2$, each $T_{\alpha}$ is an UMD space for $\alpha \in A$, $\gamma \in \R_+ \setminus \N$ and $2 \leq p,q <\infty$. Recall that the corresponding spaces $\mathcal{D}^\gamma_{p,q}$ of modelled distributions are UMD spaces and of $M$-type $2$ by Proposition~\ref{prop:UMD property}. Furthermore, let $(\Omega,\mathcal{F},\mathbb{F}, \P)$ be a complete filtered probability space and $W=(W^1,\dots,W^n)$ be $n$-dimensional Brownian motion on $[0,T]$.\smallskip

Let us consider the semilinear SPDEs defined in $\mathcal{D}^\gamma_{p,q}$:
\begin{equation}\label{eq: Brz SPDE}
  \d Y_t = - \tilde{A} Y_t \dd t + \tilde{Z}_t + \sum_{i=1}^n \tilde{B}^i Y_t \dd W^{i}_t,\quad Y_0=y_0,\quad t\in [0,T],
\end{equation}
where $\tilde{Z}$ is a function taking values in $\mathcal{D}^\gamma_{p,q}$, $\tilde{B}^i$ is a linear (generally unbounded) operator on $\mathcal{D}^\gamma_{p,q}$ for $i=1,\dots,n$, and $-\tilde{A}$ is the infinitesimal generator of an analytic semigroup $\tilde{S}(t)=\exp (-t\tilde{A})$ on $\mathcal{D}^\gamma_{p,q}$ such that:
\begin{itemize}
  \item[(H1)] There is an $M>0$ such that for all $\lambda \ge 0$, $(\tilde{A}+\lambda I)^{-1}$ exists and
    \begin{equation}\label{eq: condition 1 for Brz SPDE}
      \|(\tilde{A}+ \lambda I)^{-1}\| \le \frac{M}{1+\lambda},
    \end{equation}
    where $I$ denotes the identity map.
  \item[(H2)] For all $s \in \R$, $\tilde{A}^{\sqrt{-1} s}$ exists and is a bounded linear operator on $\mathcal{D}^\gamma_{p,q}$, $\{\tilde{A}^{\sqrt{-1} s}\}_{s\in \R}$ is a $C_0$-group on $\mathcal{D}^\gamma_{p,q}$ and for some $K>0$ and $\vartheta_{\tilde{A}} < \pi/2$ and 
  \begin{equation}\label{eq: condition 2 for Brz SPDE}
	\|\tilde{A}^{\sqrt{-1} s}\| \le K \exp(\vartheta_{\tilde{A}} |s|), \quad s \in \R.
  \end{equation}
\end{itemize} 

\begin{definition}
  A stochastic process $(Y_t)_{t\in [0,T]}$ is called \textit{mild solution} of the semilinear SPDE~\eqref{eq: Brz SPDE} if
  \begin{equation*}
     Y_t = \tilde{S}(t)y_0  +\int_0^t \tilde{S}(t-s) \tilde{Z}_s \dd s+ \sum_{i=1}^n\int_0^t \tilde{S}(t-s)\tilde{B}^i Y_s  \dd W_i(s) , \quad t\in [0,T].
  \end{equation*}
\end{definition}

Let $D_{\tilde{A}}(\frac{1}{2},2)$ be the real interpolation space between $\mathcal{D}^\gamma_{p,q}$ and $D(\tilde{A})$, the domain of $\tilde{A}$, that is
$$
  D_{\tilde{A}}(\frac{1}{2},2) := \Big\{x \in \mathcal{D}^\gamma_{p,q}: \int_0^\infty \Big|t^{\frac{1}{2}}\tilde{A}\exp\{-t\tilde{A}\}x\Big|^2 \,\frac{\d t}{t} < \infty  \Big\}.
$$
The norm on this space is given by $|x|_{D_{\tilde{A}}(\frac{1}{2},2)} := \int_0^\infty |t^{\frac{1}{2}}\tilde{A}\exp\{-t\tilde{A}\}x|^2 \,\frac{\d t}{t}$. Note that $D(\tilde{A})$ and $D_{\tilde{A}}(\frac{1}{2},2)$ are UMD spaces and of $M$-type~$2$ since $\mathcal{D}^\gamma_{p,q}$ possesses these properties, cf. \cite{Brzezniak1995}. Let us assume that $\tilde{B}^i$, $i = 1,\dots,n$, satisfies that
\begin{equation}\label{eq: condition on unbounded operator B}
  \sum_{i=1}^n |\tilde{B}^ix|^2_{D_{\tilde{A}}(\frac{1}{2},2)} \le C_1|x|^2_{D(\tilde{A})} + C_2|x|^2_{D_{\tilde{A}}(\frac{1}{2},2)} 
\end{equation}
for some constants $C_1$, $C_2$ and for all $x \in D(\tilde{A})$. Furthermore, we assume that $\tilde{Z} \in M^2(0,T;D_{\tilde{A}}(\frac{1}{2},2))$, i.e., $\int_0^T |Z_t|^2 \dd t < \infty$, and $x_0 \in L^2(D_{\tilde{A}}(\frac{1}{2},2))$. Using the same notation as in \cite{Brzezniak1995}, let $Z_T(\tilde{A})$ denote the space
$$
  Z_T(\tilde{A}) := M^2(0,T;D(\tilde{A})) \cap \mathcal{C}\Big(0,T;L^2\Big(D_{\tilde{A}}(\frac{1}{2},2)\Big)\Big).
$$
In the present setting there is an equivalent notion of solutions to the semilinear SPDE~\eqref{eq: Brz SPDE} which is called strict solution, see Definition~4.1 in \cite{Brzezniak1995}.

\begin{definition}
  A \textit{strict solution} to the semilinear SPDE~\eqref{eq: Brz SPDE} is a stochastic process $Y \in Z_T(A)$ satisfying
  $$
    Y_t + \int_0^t \tilde{A} Y_s \dd s = x_0 + \sum_{i=1}^n\int_0^t \tilde{B}^i Y_s \dd W^i_s + \int_0^t \tilde{Z}_s \dd s, \quad t\in [0,T].
  $$
\end{definition}
Indeed, it was shown in Proposition~4.2 in \cite{Brzezniak1995} that the notion of strict solutions is equivalent to the notion of mild solutions under the above stated conditions. The following corollary is a direct application of Theorem~4.6 in \cite{Brzezniak1995} to the SPDE~\eqref{eq: Brz SPDE} in the space of modelled distributions.

\begin{corollary}
  Under the conditions of the present subsection, there exists a unique mild solution $Y\in Z_T(\tilde{A})$ to the SPDE~\eqref{eq: Brz SPDE}. Equivalently, $Y$ is the unique strict solution to~\eqref{eq: Brz SPDE}.
\end{corollary}

\subsection{Mild solutions - modelled distributions and classical functions}\label{subsec:mild solutons}

In general, for a (singular) stochastic partial differential equation it is a rather delicate task to find a suitable regularity structure with an appropriate model and to set up the corresponding partial differential equation in the space of modelled distributions. 

In this section, we present how semilinear SPDEs defined in classical Besov spaces with positive regularity whose linear parts are induced by infinitesimal generators of a $C_0$-semigroup can be lifted to semilinear SPDEs in a suitable space of modelled distributions. As an application of the Fubini type theorem (Theorem~\ref{thm:fubini theorem}) and the reconstruction operator, we show that it is equivalent to solve the SPDEs in classical Besov spaces or in a certain ``abstract'' Besov space of modelled distributions.\smallskip

Let $\mathcal{B}^\gamma_{p,q}$ be the Besov space on $\R^d$ with $\gamma \in \R_+ \setminus \N$ and $1<p,q<\infty$ and thus $\mathcal{B}^\gamma_{p,q}$ satisfies the UMD property. Moreover, we suppose that $F\colon\mathcal{B}^\gamma_{p,q} \rightarrow \mathcal{B}^\gamma_{p,q}$ and $G_i\colon \mathcal{B}^\gamma_{p,q} \rightarrow \mathcal{B}^\gamma_{p,q}$, $i=1,\ldots,n$, are measurable maps, and $W = (W^1, \ldots,W^n)$ is an $\R^n$-valued Brownian motion defined on a complete filtered probability space $(\Omega, \mathcal{F}, \mathbb{F},\P)$. 

We consider the semilinear SPDE in the classical Besov space
\begin{equation}\label{eq: semilinear SPDE}
  \d X_t = (-AX_t + F(X_t)) \dd t + \sum_{i=1}^n G_i(X_t) \dd W^i_t, \quad X_0 = x_0 \in \mathcal{B}^\gamma_{p,q},
\end{equation}
and recall that $(X_t)_{t\in [0,T]}$ is a \textit{c\`adl\`ag mild solution} if it satisfies
\begin{equation}\label{eq: mild solution} 
  X_t = S(t)X_0 + \int_0^t S(t-s)F(X_s)\dd s + \sum_{i=1}^n\int_0^t S(t-s)G_i(X_s) \dd W^i_s,
\end{equation}
for all $t \in [0,T]$, and $(X_t)_{t\in [0,T]}$ has almost surely c\`adl\`ag (i.e. right-continuous with left limits) sample paths, where $S(t)$ denotes the $C_0$-semigroup of bounded linear operators on $\mathcal{B}^\gamma_{p,q}$ generated by the operator~$-A$. Note that the stochastic integral in \eqref{eq: mild solution} is well-defined since the underlying Banach space $\mathcal{B}^\gamma_{p,q}$ satisfies the UMD property. \smallskip

In order to obtain the semilinear SPDE in a space of modelled distributions corresponding to~\eqref{eq: semilinear SPDE} we choose the polynomial regularity structure $\overline{\mathcal{T}}$ on $\R^d$ and the corresponding polynomial model as defined in Examples~\ref{ex:polynomial} and~\ref{ex:polymodel}. Recall that by Proposition~\ref{prop:UMD property} the space $\mathcal{D}^\gamma_{p,q} = \mathcal{D}^\gamma_{p,q}(\overline{\mathcal{T}})$ of modelled distributions is an UMD space. \smallskip

Assuming $\gamma>0$, recall that there exists a unique reconstruction operation $\mathcal{R}$ acting on $\mathcal{D}^\gamma_{p,q}$, precisely defined in Theorem~3.1 \cite{Hairer2017}, cf. Corollary~\ref{cor:reconstruction theorem with local models}. Moreover, the reconstruction operator $\mathcal{R}$ is continuous isomorphism between $\mathcal{D}^\gamma_{p,q}$ and $\mathcal{B}^\gamma_{p,q}$, see Proposition~3.4 in \cite{Hairer2017}.

In order to lift the semilinear SPDE~\eqref{eq: semilinear SPDE} to an equivalent SPDE in the space of modelled distributions, we introduce the conjugation mapping $C_{\mathcal{R}}$ from $\mathcal{L}(\mathcal{B}^\gamma_{p,q})$, the space of all bounded linear operators on $\mathcal{B}^\gamma_{p,q}$, onto $\mathcal{L}(\mathcal{D}^\gamma_{p,q})$, induced by $\mathcal{R}$:
\begin{equation*}
  C_{\mathcal{R}}(L) := \mathcal{R}^{-1} \circ L \circ \mathcal{R}.
\end{equation*}
Notice that $C_{\mathcal{R}}$ is a continuous isomorphism and, by verifying the corresponding definitions, we obtain that:
\begin{enumerate}
  \item[(i)] $S(t)$ is a $C_0$-semigroup on $\mathcal{B}^\gamma_{p,q}$ if and only if 
    $$
      T(t) := C_{\mathcal{R}}(S(t)) =\mathcal{R}^{-1} \circ S(t) \circ \mathcal{R}
    $$ is a $C_0$-semigroup on  $\mathcal{D}^\gamma_{p,q}$.
  \item[(ii)] $-A$ is the infinitesimal generator of a $C_0$-semigroup $S(t)$ on $\mathcal{B}^\gamma_{p,q}$ if and only if $C_{\mathcal{R}}(-A) := \mathcal{R}^{-1} \circ (-A) \circ \mathcal{R}$ is the infinitesimal generator of a $C_0$-semigroup $T(t) = C_{\mathcal{R}}(S(t))$ on $\mathcal{D}^\gamma_{p,q}$.
  \item[(iii)] Let $-A$ be the infinitesimal generator of a $C_0$-semigroup $S(t)$ on $\mathcal{B}^\gamma_{p,q}$. Then the resolvent set $\rho(-A)$ of $-A$ is equal to the resolvent set of $C_{\mathcal{R}}(-A)$. Moreover, for every $\lambda \in \rho(-A)$, the resolvents $R(\lambda:-A)$ and $R(\lambda:C_{\mathcal{R}}(-A))$ satisfies that $R(\lambda:C_{\mathcal{R}}(-A))=C_{\mathcal{R}}(R(\lambda:-A))$. The converse also holds true.
\end{enumerate}

Due to these elementary facts and the Fubini type theorem (Theorem~\ref{thm:fubini theorem}), we can establish the following equivalence result for mild solution of semilinear SPDEs in classical Besov spaces and Besov spaces of modelled distributions, respectively.

\begin{theorem}\label{thm: application of Fubini}
  Let $-A$ be the infinitesimal generator of a $C_0$-semigroup $S(t)$ on $\mathcal{B}^\gamma_{p,q}$. If $(X_t)_{t\in [0,T]}$ is a c\`adl\`ag mild solution to the semilinear SPDE \eqref{eq: semilinear SPDE}, then $Y_t:=\mathcal{R}^{-1}(X_t)$, $t\in [0,T]$, is a c\`adl\`ag mild solution to the semilinear SPDE on $\mathcal{D}^\gamma_{p,q}$
  \begin{equation}\label{eq: SPDE on modelled distribution space}
    \d Y_t = (-\tilde{A}Y_t + \tilde{F}(Y_t))\dd t + \sum_{i=1}^n\tilde{G_i}(Y_t)\dd W^i_t, \quad Y_0 = \mathcal{R}^{-1}x_0 \in \mathcal{D}^\gamma_{p,q},
  \end{equation}
  where $-\tilde{A} = C_{\mathcal{R}}(-A) := \mathcal{R}^{-1} \circ (-A) \circ \mathcal{R}$, $\tilde{F}:= \mathcal{R}^{-1} \circ F \circ \mathcal{R}$ and $\tilde{G_i} := \mathcal{R}^{-1} \circ G_i \circ \mathcal{R}$.
  
  Conversely, if $(Y_t)_{t\in [0,T]}$ is a c\`adl\`ag mild solution to the semilinear SPDE~\eqref{eq: SPDE on modelled distribution space} on $\mathcal{D}^\gamma_{p,q}$ related to the infinitesimal generator $-\tilde{A}$ of a $C_0$-semigroup $T(t)$ on $\mathcal{D}^\gamma_{p,q}$, vector fields $\tilde{F}$ and $\tilde{G_i}$ and initial value $Y_0 \in \mathcal{D}^\gamma_{p,q}$, then $X_t:=\mathcal{R}(Y_t)$, $t\in [0,T]$, is a c\`adl\`ag mild solution to the equation~\eqref{eq: semilinear SPDE} with $-A = C_{\mathcal{R}^{-1}}(-\tilde{A}) = \mathcal{R} \circ (-\tilde{A}) \circ \mathcal{R}^{-1}$, $F = \mathcal{R} \circ \tilde{F} \circ \mathcal{R}^{-1}$, $G_i = \mathcal{R} \circ \tilde{G_i} \circ \mathcal{R}^{-1}$ and $X_0 = \mathcal{R}(Y_0)$.
\end{theorem}

\begin{proof}
  We will prove the second part, the first assertion follows by using a similar argument.
  
  Suppose that $-\tilde{A}$ is the infinitesimal generator of a $C_0$-semigroup $T(t)$ of bounded linear operators on $\mathcal{D}^\gamma_{p,q}$ and $(Y_t)_{t\in [0,T]}$ is a c\`adl\`ag mild solution to the SPDE~\eqref{eq: SPDE on modelled distribution space} associated with vector fields $\tilde{F}$ and $\tilde{G_i}$, that is, 
  $$
    Y_t = T(t)Y_0 + \int_0^t T(t-s)\tilde{F}(Y_s) \dd s  + \sum_{i=1}^n\int_0^t T(t-s)\tilde{G_i}(Y_s) \dd W^i_s, \quad t\in [0,T].
  $$
  Let $S(t) := C_{\mathcal{R}^{-1}}(T(t)) = \mathcal{R} \circ T(t) \circ \mathcal{R}^{-1}$. As we have checked, it is a $C_0$-semigroup of bounded linear operators on $\mathcal{B}^\gamma_{p,q}$ with the generator $-A = \mathcal{R} \circ (-\tilde{A}) \circ \mathcal{R}^{-1}$. For any smooth compactly supported test function $\psi$, using Theorem~\ref{thm:fubini theorem}, we deduce that $X_t := \mathcal{R}(Y_t)$ satisfies the following dynamics (with $F = \mathcal{R} \circ \tilde{F} \circ \mathcal{R}^{-1}$, $G_i = \mathcal{R} \circ \tilde{G_i} \circ \mathcal{R}^{-1}$ and $X_0 = \mathcal{R}(Y_0)$):
  \begin{align*}
    \langle \mathcal{R}(Y_t), \psi \rangle 
    &=\Big  \langle \mathcal{R}\Big(T(t)Y_0\Big), \psi \Big\rangle + \bigg\langle \int_0^t \mathcal{R}\Big(T(t-s)\tilde{F}(Y_s)\Big)\dd s, \psi \bigg\rangle \\
    & \quad \quad + \sum_{i=1}^n\int_0^t\Big \langle \mathcal{R}\Big(T(t-s)\tilde{G_i}(Y_s)\Big), \psi \Big\rangle \dd W^i_s \\
    &= \langle S(t)X_0, \psi \rangle + \bigg\langle \int_0^t S(t-s)F(X_s)\dd s, \psi\bigg \rangle + \sum_{i=1}^n\int_0^t \langle S(t-s)G_i(X_s), \psi \rangle \dd W^i_s,
  \end{align*}
  for $t\in [0,T]$. Since the dual pairing $\langle \cdot, \psi \rangle$ with respect to a test function $\psi$ is a continuous linear functional on $\mathcal{B}^\gamma_{p,q}$, the above observation ensures that $X_t$ satisfies that
  $$
    \langle X_t, \psi \rangle = \langle S(t)X_0, \psi \rangle + \bigg\langle \int_0^t S(t-s)F(X_s)\dd s, \psi \bigg\rangle + \sum_{i=1}^n\bigg\langle \int_0^t  S(t-s)G_i(X_s) \dd W^i_s, \psi \bigg\rangle,
  $$
  where $\int_0^t  S(t-s)G_i(X_s) \dd W^i_s$ is the stochastic integral defined on the UMD space $\mathcal{B}^\gamma_{p,q}$. Since the above equation holds for any test function $\psi$ and $(Y_t)_{t\in [0,T]}$ is c\`adl\`ag, we indeed have
  $$
    X_t = S(t)X_0 + \int_0^t S(t-s)F(X_s)\dd s + \sum_{i=1}^n\int_0^t S(t-s)G_i(X_s) \dd W^i_s,\quad t\in [0,T],
  $$
  which shows that $(X_t)_{t\in [0,T]}$ is a c\`adl\`ag mild solution to the SPDE~\eqref{eq: semilinear SPDE}.
\end{proof}

\subsection{Strict solutions - modelled distributions and classical functions}

Coming back to the semilinear SPDE~\eqref{eq: Brz SPDE}, an equivalence result analogously to Theorem~\ref{thm: application of Fubini} holds also for strict solutions of semilinear SPDEs in classical Besov spaces and Besov spaces of modelled distributions, respectively. In this subsection we again work with the polynomial regularity structure and keep the setting of Subsection~\ref{subsec:mild solutons}.\smallskip

For simplicity, we consider the Besov space $\mathcal{B}^{\gamma}_{p,2}(\R^d)$ with $\gamma >0$, $2\le p <\infty$ and the semilinear SPDEs defined in $\mathcal{B}^{\gamma}_{p,2}(\R^d)$:
\begin{equation}\label{eq:brz besov space}
  \d X_t = - A X_t \dd t + Z_t + \sum_{i=1}^n B^i X_t \dd W^{i}_t,\quad X_0=x_0,\quad t\in [0,T],
\end{equation}
where $Z$ is a function taking values in $\mathcal{B}^{\gamma}_{p,2}(\R^d)$, $B^i$ is a linear (generally unbounded) operator for $i=1,\dots,n$, and $-\tilde{A}$ is the infinitesimal generator of an analytic semigroup $S(t)=\exp (-t\tilde{A})$ on $\mathcal{B}^{\gamma}_{p,2}(\R^d)$ satisfying (H1) and (H2). Note that $\mathcal{B}^{\gamma}_{p,2}(\R^d)$ 
is a Banach space having the $M$-type $2$ and UMD properties, see Corollary~A.6 in \cite{Brzezniak1995}. All the ingredients fulfill analogous conditions as stated in Subsection~\ref{subsec:existence of SPDE} or \cite{Brzezniak1995}, respectively.
 
We will show that if $(X_t)_{t\in [0,T]}$ is a strict solution to the SPDE~\eqref{eq:brz besov space} on $\mathcal{B}^{\gamma}_{p,2}$, then $Y_t:=\mathcal{R}^{-1}X_t$, $t\in [0,T]$, is a strict solution to SPDE~\eqref{eq: Brz SPDE} on $\mathcal{D}^\gamma_{p,2}$ of the form
\begin{equation}\label{eq: Brz SPDE in modelled distribution spaces}
  \d Y_t = -\tilde{A}Y_t \dd t + \tilde{Z}_t + \sum_{i=1}^n \tilde{B^i}Y_t \dd W^i_t, \quad Y_0 = \mathcal{R}^{-1}x_0,
\end{equation}
where $\tilde{A} := C_{\mathcal{R}}(A) = \mathcal{R}^{-1} \circ A \circ \mathcal{R}$, $\tilde{Z}_t:= \mathcal{R}^{-1}Z_t$ and $\tilde{B^i} = C_{\mathcal{R}}(B^i) = \mathcal{R}^{-1} \circ B^i \circ \mathcal{R}$.

To this end, let us first check that $\tilde{A}$, $\tilde{Z}$ and $\tilde{B^i}$ satisfy the same properties on $\mathcal{D}^\gamma_{p,2}$ as for $A$, $f$ and $B^i$ on $\mathcal{B}^\gamma_{p,2}$. First notice that, if $Z$ satisfies that $\int_0^t |Z_s|^2 \dd s< \infty$, then the boundedness of $\mathcal{R}^{-1}$ will ensure that $\tilde{Z} = \mathcal{R}^{-1}Z$ is also square integrable.

Let us now focus on the linear operator $\tilde{A}$. As we have noticed, the resolvent $R(\lambda:-\tilde{A})$ satisfies that $C_{\mathcal{R}}(R(\lambda:-A))$. It follows that for any $n \ge 1$, we have $R(\lambda:-\tilde{A})^{n} = C_{\mathcal{R}}(R(\lambda:-A)^n)$ and $\tilde{A}R(\lambda:-\tilde{A})^{n} = C_{\mathcal{R}}(AR(\lambda:-A)^n)$ on $D(\tilde{A}) = \mathcal{R}^{-1}(A)$. Therefore, since $-A$ is the infinitesimal generator of an analytic semigroup $\exp\{-tA\}$ on $\mathcal{B}^\gamma_{p,2}$, using \cite[Theorem~5.5]{Pazy1983} we can immediately deduce that $-\tilde{A}$ generates an analytic semigroup, say, $\exp\{-t\tilde{A}\}$, on the modelled distribution space $\mathcal{D}^\gamma_{p,2}$. The same argument implies that if for all $\lambda \ge 0$, $(A+\lambda I)^{-1}$ exists and satisfies the bound \eqref{eq: condition 1 for Brz SPDE}, then for all $\lambda \ge 0$, $(\tilde{A} + \lambda I)^{-1}$ exists (note that $(\tilde{A} + \lambda I)^{-1} = R(\lambda;-\tilde{A})$) and it satisfies the bound \eqref{eq: condition 1 for Brz SPDE} with $M$ replaced by $M\|\mathcal{R}\|\|\mathcal{R}^{-1}\|$. Repeating this procedure once again we can show that for all $s \in \R$, $\tilde{A}^{is}$ form a $C_0$-group on $\mathcal{D}^\gamma_{p,2}$ and the bound \eqref{eq: condition 2 for Brz SPDE} holds with $K$ replaced by $K\|\mathcal{R}\|\|\mathcal{R}^{-1}\|$; and $D_{\tilde{A}}(\frac{1}{2},2) = \mathcal{R}^{-1}(D_A(\frac{1}{2},2))$. As a consequence, it holds that the linear operators $\tilde{B^i}$ satisfy the bound \eqref{eq: condition on unbounded operator B} as for $B^i$'s but with different constants determined by $\|\mathcal{R}\|$ and $\|\mathcal{R}\|^{-1}$. In particular, we conclude that 
\begin{equation*}
  Z_T(\tilde{A}) = \mathcal{R}^{-1}(Z_T(A)).
\end{equation*}

Now we are in the position to prove the following result:

\begin{proposition}
  Let $A$, $f$ and $B^i$, $i=1,\ldots,n$, be defined on $\mathcal{B}^\gamma_{p,2}$ and satisfy all above conditions listed in Subsection~\ref{subsec:existence of SPDE}. Let $\tilde{A}$, $\tilde{f}$ and $\tilde{B^i}$, $i=1,\ldots,n$, be defined as above. Then, $X \in Z_T(A)$ is a strict solution to semilinear SPDE~\eqref{eq:brz besov space} associated with $(A,f,(B^i))$ if and only if $Y:=\mathcal{R}^{-1}X \in Z_T(\tilde{A})$ is a strict solution to semilinear SPDE~\eqref{eq: Brz SPDE in modelled distribution spaces} on $\mathcal{D}^\gamma_{p,2}$ associated with $(\tilde{A},\tilde{f},(\tilde{B^i}))$.
\end{proposition}

\begin{proof}
  Let $X \in Z_T(A)$ be a strict solution to problem~\eqref{eq:brz besov space} associated with $(A,f,(B^i))$. Then as we have shown, $\mathcal{R}^{-1}X$ takes values in $Z_T(\tilde{A})$. By Theorem~\ref{thm: application of Fubini}, it holds that $\mathcal{R}^{-1}X$ is a mild solution to problem \eqref{eq: Brz SPDE in modelled distribution spaces} on $\mathcal{D}^\gamma_{p,2}$ associated with $(\tilde{A},\tilde{f},(\tilde{B^i}))$. Since we have checked that $(\tilde{A},\tilde{f},(\tilde{B^i}))$ satisfy all conditions listed in \cite{Brzezniak1995} on the modelled distribution space $\mathcal{D}^\gamma_{p,2}$ as the conditions fulfilled by $(A,f,(B^i))$ on the Besov space $\mathcal{B}^\gamma_{p,2}$, using Proposition~4.2 in \cite{Brzezniak1995} we can conclude that $\mathcal{R}^{-1}X$ is also a strict solution to problem~\eqref{eq: Brz SPDE in modelled distribution spaces} on $\mathcal{D}^\gamma_{p,2}$ associated with $(\tilde{A},\tilde{f},(\tilde{B^i}))$. The converse can be proved by using the same reasoning, just by noticing that $A = C_{\mathcal{R}^{-1}}(\tilde{A}) = \mathcal{R} \circ \tilde{A} \circ \mathcal{R}^{-1}$.
\end{proof}

\section{Distribution-valued It{\^o} stochastic differential equations}\label{sec: distribution valued SDEs}

As a toy example of distribution-valued stochastic differential equations, we consider
\begin{equation}\label{eq: distribution-valued SDE}
  \d Y_t = (Y_t \cdot \phi) \dd W_t, \quad Y_0 = \xi,
\end{equation}
for $\phi \in \mathcal{C}^\beta(\mathcal{K})$ with $\beta > 0$ and $\xi \in \mathcal{C}^\alpha(\mathcal{K})$ with $\alpha<0$ satisfying $\alpha + \beta > 0$, where $\mathcal{K} \subset \R^d$ is a compact set and $W= (W_t)_{t \in [0,T]}$ is a Brownian motion. Recall that $\mathcal{C}^\beta(\mathcal{K}):=\mathcal{B}^\beta_{\infty,\infty}(\mathcal{K})$ stands for the usual space of $\beta$-H\"older continuous functions restricted to the compact set $\mathcal{K}$. The product mapping $\cdot$ denotes the continuous bilinear map from $\mathcal{C}^\alpha(\mathcal{K}) \times \mathcal{C}^\beta(\mathcal{K})$ into $\mathcal{C}^\alpha(\mathcal{K})$ which continuously extends the classical product of smooth functions $(f,g) \mapsto fg$, see e.g. \cite[Proposition~4.14]{Hairer2014}.

While the linear SDE~\eqref{eq: distribution-valued SDE} might look rather simple at the first glimpse, neither an approach based on stochastic integration nor a rough path based approach directly provides an existence and uniqueness result for~\eqref{eq: distribution-valued SDE}. Indeed, since $\mathcal{C}^\alpha(\mathcal{K})$ is neither an UMD Banach space nor of martingale type $2$, the It\^o integral $\int_0^\cdot (Y_t \cdot \phi)\dd W_t$ is not defined via classical stochastic integration theory (e.g. \cite{vanNeerven2007}). To overcome, this issue one could embed the H\"older space $\mathcal{C}^\alpha(\mathcal{K})$ into a suitable Sobolev space and consider \eqref{eq: distribution-valued SDE} as an equation on Soblev spaces. However, in this way the obtained solution a priori has not optimal regularity in space, which is expected to be $\mathcal{C}^\alpha(\mathcal{K})$ for each $t$. Alternatively, we may interpret equation~\eqref{eq: distribution-valued SDE} as an random rough differential equation (RRDE) with respect to the It\^o rough path $\mathbf{W} = (W,\mathbb{W})$ with $\mathbb{W}^{i,j}_{s,t} = \int_s^t (W_r^i - W^i_s)\dd W^j_r$ being the It\^o integral. However, since the underlying vector field $\eta \mapsto \phi \cdot \eta$ is a bounded linear mapping, we can only conclude that there exists a unique \textit{local} solution to the RRDE~\eqref{eq: distribution-valued SDE} by using classical rough path theory (e.g. \cite[Chapter~8]{Friz2014}). In fact, global existence and uniqueness results for rough differential equations (RDEs) with linear vector fields are a delicate challenge. Let us mention that linear RDEs have been considered by several authors since they are an essential tool for studying the derivative of the It\^o map and its flow properties. In particular, \cite{Lyons1998} shows that linear RDEs driven by geometric rough paths admit global unique solutions, but not covering non-geometric rough paths like It\^o rough paths. A similar result can be found in~\cite{Friz2010} for linear RDEs driven by weakly geometric rough paths in finite dimension. For further results in this direction and references, see, e.g., \cite{Lejay09}. 

In the following we provide a unique \textit{global} solution for the stochastic differential equation~\eqref{eq: distribution-valued SDE} with optimal H\"older regularity in space based on a ``mixed'' approach using regularity structures and stochastic integration in UMD Banach spaces.

As a first step, we formulate a differential equation, corresponding to the SDE~\eqref{eq: distribution-valued SDE}, on the space of modelled distributions for a specific regularity structure. This relies on the observation that the product mapping $\cdot$ from $\mathcal{C}^\alpha \times \mathcal{C}^\beta$ into $\mathcal{C}^\alpha$ can be constructed in an elegant way via regularity structures, see Proposition~4.14 in \cite{Hairer2014}. Following the construction in the proof of Proposition~4.14 in \cite{Hairer2014}, we build a regularity structure $\mathcal{T}^\xi = (A,T,G)$ such that
\begin{itemize}
  \item $A = \N \cup (\N+\alpha)$,
  \item $T = V \oplus U$, where $V = \mathcal{T}_d$ is the canonical polynomial model space (see Example~\ref{ex:polynomial}) spanned by monomials $X^k$, $k \in \N^d$, and $U$ is isomorphic to $V$ but with canonical basis $\Xi X^k$,
  \item $G$ is the canonical structure group for the polynomial model space $\mathcal{T}_d$.
\end{itemize}
Moreover, we define a model $(\Pi^\xi,\Gamma)$ for $\mathcal{T}^\xi$ by
\begin{itemize}
  \item $(\Pi^\xi_x X^k)(y) = (y-x)^k$ and $(\Pi^\xi_x \Xi X^k)(y) = (y-x)^k\xi(y)$,
  \item $\Gamma_{x,y}X^k = (X + x-y)^k$ and $\Gamma_{x,y}\Xi X^k = \Xi(X + x-y)^k$,
\end{itemize}
for all $x,y \in \R^d$. The product $*$ between $V$ and $U$ is given by the natural identity
\begin{equation*}
  (\Xi X^k) * (X^l) = \Xi X^{k+l}.
\end{equation*}

Let $\mathcal{R}^\xi$ be the reconstruction operator in \cite[Theorem~3.1]{Hairer2017} defined on space of modelled distributions $\mathcal{D}^\gamma_{p,q}=\mathcal{D}^\gamma_{p,q}(\mathcal{T}^\xi)$ for $\gamma > 0$, $p,q\in[1,\infty]$. Since the constant map $\Xi$ belongs to $\mathcal{D}^\gamma_{\infty,\infty}$ for all $\gamma > 0$ and $\Pi_x^\xi(\Xi) = \xi$ holds for all $x \in \R^d$, the uniqueness property of $\mathcal{R}^\xi$ implies that $\mathcal{R}^\xi(\Xi) = \xi$. Finally, by Proposition~3.4 in \cite{Hairer2017} or \cite[(2.6)]{Hairer2014} there exists a $\tilde{\phi} \in \mathcal{D}^\beta_{\infty,\infty}$ such that $\mathcal{R}^\xi(\tilde{\phi}) = \phi$.

The counterpart of~\eqref{eq: distribution-valued SDE} on the space of modelled distributions is 
\begin{equation}\label{eq: RDE lifted to modelled distribution spaces}
  \d Z_t = (Z_t * \tilde{\phi}) \dd W_t, \quad Z_0 = \Xi, \quad t\in [0,T].
\end{equation}
In a similar spirit as for the semilinear SPDEs of Section~\ref{sec:SPDEs}, we first obtain a unique global solution~$Z$ using stochastic integration on the space of modelled distributions. Applying the reconstruction theorem and the Fubini type theorem (Theorem~\ref{thm:fubini theorem}) we deduce the existence of a unique global solution $Y$ of the stochastic differential equation~\eqref{eq: distribution-valued SDE}.

\begin{theorem}
  Let $\alpha<0$, $\beta>0$ be such that $\alpha + \beta > 0$, $\phi \in \mathcal{C}^\beta(\mathcal{K})$, $\xi \in \mathcal{C}^\alpha(\mathcal{K})$, and $(\mathcal{T}^\xi,(\Pi^\xi,\Gamma))$ be the regularity structure induced by $\xi$ defined as above. Let $\tilde{\phi} \in \mathcal{D}^\beta_{\infty,\infty}$ be the modelled distribution for $(\mathcal{T}^\xi,(\Pi^\xi,\Gamma))$ satisfying $\mathcal{R}^\xi(\tilde{\phi}) = \phi$. Then, there exists a unique solution $Z$ taking values in $\mathcal{D}^{\alpha + \beta}_{\infty,\infty}(U)$ to the differential equation~\eqref{eq: RDE lifted to modelled distribution spaces} and the process $Y_t := \mathcal{R}^\xi(Z_t)$, $t\in [0,T]$, is the unique solution to the SDE~\eqref{eq: distribution-valued SDE} defined on the H\"older space~$\mathcal{C}^\alpha(\mathcal{K})$.
\end{theorem}

\begin{proof}
  To simplify the notation we omit $\mathcal{K}$ in the following and w.l.o.g. we assume that $\beta \in (0,1)$ as the general case can be handled in an analogous way. The lifted modelled distribution $\tilde{\phi} \in \mathcal{D}^\beta_{\infty,\infty}$ of $\phi$ can be written as $\tilde{\phi}(x) = \phi(x)\1$, where $\1$ is the basis of $V_0 \simeq \R$, and $\phi$ is H\"older continuous of regularity $\beta$. The norm of $\phi$ in $\mathcal{C}^\beta$ is denoted by $\|\phi\|_{\beta}$ and given by
  \begin{equation*}
    \|\phi\|_{\beta} := \|\phi\|_{\infty} + |\phi|_{\beta}, \quad |\phi|_{\beta}:= \inf\big\{C \ge 0: |\phi(x) - \phi(y)|\le C|x-y|^\beta\,\,\, \forall x,y\in \mathcal{K} \big\}.
  \end{equation*}
  
  \textit{Step~1:} Let $U_\alpha := \text{span}(\Xi) \subset U$. Note that each $f \in \mathcal{D}^{\alpha + \beta}_{\infty,\infty}(U_\alpha)$ can  be represented by $f(x) = f_\alpha(x)\Xi$, where $f_\alpha \in \mathcal{C}^\beta$ and $\vertiii{f}_{\alpha+\beta,\infty,\infty} = \|f_\alpha\|_{\beta}$. In other words, we have $\mathcal{D}^{\alpha + \beta}_{\infty,\infty}(U_\alpha) \simeq \mathcal{C}^\beta$ is a Banach space. Then we have 
  \begin{equation*}
    f(x) * \tilde{\phi}(x) = f_\alpha(x)\phi(x) \Xi\quad \text{and} \quad \vertiii{f*\tilde{\phi}}_{\alpha+\beta,\infty,\infty} \le \|\phi\|_{\beta}\vertiii{f}_{\alpha+\beta,\infty,\infty},
  \end{equation*}
  and thus the linear vector field from $\mathcal{D}^{\alpha + \beta}_{\infty,\infty}(U_\alpha)$ into itself: $f\mapsto f*\tilde{\phi}$
  is continuous. Since $\Xi$ as a constant map belongs to $\mathcal{D}^{\alpha+ \beta}_{\infty,\infty}(U_\alpha)$ as well, the standard results from rough path theory (see e.g. \cite{Friz2014}) ensure that RRDE~\eqref{eq: RDE lifted to modelled distribution spaces} (w.r.t. the It\^o lift of Brownian motion) admits a unique solution $Z_t$ taking values in $\mathcal{D}^{\alpha+ \beta}_{\infty,\infty}(U_\alpha)$ with $Z_0 = \Xi$ on some subinterval $[0,t_1)$ with $t_1 \le T$. Moreover, $Z_t$ has $\rho$-H\"older continuous sample paths in time $t \in [0,t_1)$ for any $\rho < 1/2$. 
  
  Next let us consider the process $Y_t := \mathcal{R}^\xi(Z_t)$. By the reconstruction theorem (\cite[Theorem~3.10]{Hairer2014}), $Y_t$ is uniquely determined by $Z_t$ and takes values in $\mathcal{C}^\alpha(\R^d)$. Moreover, by the It\^o's formula for rough path integration applied to the bounded linear mapping $\mathcal{R}^\xi$, see Proposition~5.6 in \cite{Friz2014}, one has that $Y_t$, $t \in [0,t_1)$ satisfies 
  \begin{equation}\label{eq: the RDE in function spaces}
    Y_t = Y_0 + \int_0^t \mathcal{R}^\xi(Z_t * \tilde{\phi}) \dd \mathbf{W}_t, \quad Y_0 = \mathcal{R}^\xi(\Xi) = \xi.
  \end{equation}

  \textit{Step~2}: In this step we will show that $Y_t = \mathcal{R}^\xi(Z_t)$ solves the stochastic differential equation~\eqref{eq: distribution-valued SDE}, as long as $Z_t$ is a solution to \eqref{eq: RDE lifted to modelled distribution spaces}. In view of equation~\eqref{eq: the RDE in function spaces}, it suffices to show that for each $t$ such that $Z_t$ is well-defined, one has $\mathcal{R}^\xi(Z_t * \tilde{\phi}) = \mathcal{R}^\xi(Z_t) \cdot \phi = Y_t \cdot \phi$.
  
  Let us first recall how the product $\cdot$ is constructed via regularity structures. In view of Proposition~4.14 in \cite{Hairer2014}, for each $\mathcal{R}^\xi(Z_t)=:\xi_t\in \mathcal{C}^\alpha(\R^d)$ we denote by $\Xi_t$ the formal basis of the level-$\alpha$ space of the regularity structure $(\mathcal{T}^{\xi_t},(\Pi^{\xi_t},\Gamma))$ which is defined exactly as $(\mathcal{T}^{\xi},(\Pi^{\xi},\Gamma))$ by replacing $\Xi$ through $\Xi_t$, then it holds that
  \begin{equation*}
    \mathcal{R}^\xi(Z_t) \cdot \phi = \mathcal{R}^{\xi_t} (\Xi_t * \tilde{\phi}),
  \end{equation*}
  where $\mathcal{R}^{\xi_t}$ is the reconstruction operator associated with  $(\mathcal{T}^{\xi_t},(\Pi^{\xi_t},\Gamma))$. Now we need to check that $\mathcal{R}^\xi(Z_t * \tilde{\phi}) = \mathcal{R}^{\xi_t} (\Xi_t * \tilde{\phi})$ for all $t \in [0,T]$.
  
  Towards this aim, we first note that both $Z_t * \tilde{\phi}$ and $\Xi_t * \tilde{\phi}$ are in $\mathcal{D}^{\alpha+ \beta}_{\infty,\infty}$ for respective regularity structures, with $\gamma := \alpha + \beta > 0$ the reconstruction theorem (Theorem~3.10 in\cite{Hairer2014}) tells us that the following bounds are valid:
  \begin{align*}
    \Big|  \langle \mathcal{R}^\xi(Z_t * \tilde{\phi}) - \Pi^\xi_x(Z_t * \tilde{\phi})(x) , \eta^\lambda_x \rangle    \Big| \lesssim \lambda^\gamma \quad\text{and}\quad
    \Big|  \langle \mathcal{R}^{\xi_t} (\Xi_t * \tilde{\phi}) - \Pi^{\xi_t}_x(\Xi_t * \tilde{\phi})(x) , \eta^\lambda_x \rangle    \Big| \lesssim \lambda^\gamma,
  \end{align*}
  where proportionality constants are uniform in $x \in \R^d$, $\lambda \in (0,1]$ and $\eta \in \mathcal{B}^r$ for $r > |\alpha|$ but may depend on the norms of $(\Pi^\xi,\Gamma)$, $(\Pi^{\xi_t},\Gamma)$, $Z_t * \tilde{\phi}$ or $\Xi_t * \tilde{\phi}$. On the other hand, if we write $Z_t(x) = Z^\alpha_t(x)\Xi$, then it holds that $(Z_t * \tilde{\phi})(x) = Z^\alpha_t(x)\phi(x)\Xi$, and therefore $\Pi^{\xi}_x(Z_t * \tilde{\phi})(x) = Z^\alpha_t(x)\phi(x)\xi$. The same reasoning yields that $\Pi^{\xi_t}_x(\Xi_t * \tilde{\phi})(x) = \phi(x)\xi_t$. However, since $\xi_t = \mathcal{R}^\xi(Z_t)$, the reconstruction theorem applied to $\mathcal{R}^\xi(Z_t)$ gives that
  $$
    \Big|  \langle \mathcal{R}^\xi(Z_t) - Z^\alpha_t(x)\xi , \eta^\lambda_x \rangle    \Big| \lesssim \lambda^\gamma,
  $$
  where we used the fact that $\Pi^\xi_x(Z_t)(x) = Z^\alpha_t(x)\xi$. To summarize, now we obtain that 
  \begin{align*}
    \Big|  \langle \Pi^\xi_x(Z_t * \tilde{\phi})(x) - \Pi^{\xi_t}_x(\Xi_t * \tilde{\phi})(x) , \eta^\lambda_x \rangle    \Big| &= \Big|  \langle \phi(x)(\mathcal{R}^\xi(Z_t) - Z^\alpha_t(x)\xi) , \eta^\lambda_x \rangle    \Big|\\
    &\le \|\phi\|_\infty \Big|  \langle \mathcal{R}^\xi(Z_t) - Z^\alpha_t(x)\xi , \eta^\lambda_x \rangle    \Big| \lesssim \lambda^\gamma.
  \end{align*}
  Combining all the above bounds, we deduce that
  $$
    \Big|  \langle \mathcal{R}^\xi(Z_t * \tilde{\phi}) - \mathcal{R}^{\xi_t} (\Xi_t * \tilde{\phi}) , \eta^\lambda_x \rangle \Big| \lesssim \lambda^\gamma,
  $$
  uniformly in $x \in \R^d$, $\lambda \in (0,1]$ and $\eta \in \mathcal{B}^r$ for $r > |\alpha|$. Then, since $\gamma > 0$, the same argument used in the proof of the uniqueness of reconstruction operator, see \cite[Theorem~3.10]{Hairer2014} shows that $\mathcal{R}^\xi(Z_t * \tilde{\phi}) = \mathcal{R}^{\xi_t} (\Xi_t * \tilde{\phi})$ for all $t \in [0,T]$, as we claimed at the beginning of this part. As a consequence, $Y_t = \mathcal{R}^\xi(Z_t)$ solves~\eqref{eq: distribution-valued SDE}. 
  
  \textit{Step~3:} Let $\infty >p >1$. Due to the compactness of $\mathcal{K}$, the constant mapping $\Xi$ belongs to $\mathcal{D}^{\alpha + \bar{\beta}}_{p,p}(U_\alpha)$ for any $\bar{\beta} < \beta$ and $\bar{\beta} + \alpha > 0$. Since $\mathcal{D}^{\alpha+\bar{\beta}}_{p,p}(U_\alpha) \simeq \mathcal{B}^{\bar{\beta}}_{p,p}$ is an UMD Banach space, we are able to consider equation~\eqref{eq: RDE lifted to modelled distribution spaces} as a stochastic differential equation on $\mathcal{D}^{\alpha+\bar{\beta}}_{p,p}(U_\alpha)$. 
  
  Let us consider the multiplication map $f \mapsto f*\tilde{\varphi}$ for $f \in \mathcal{D}^{\alpha+\bar{\beta}}_{p,p}(U_\alpha)$. As for the H\"olderian case, $f$ can be written as $f(x) = f_\alpha(x) \Xi$ and $f(x) * \tilde{\varphi}(x) = f_\alpha(x)\phi(x) \Xi$. As a consequence, we obtain that
  $$
    \vertiii{f*\tilde{\varphi}}_{\bar{\gamma},p,p} = \Big(\int_{h \in B(0,1)} \Big\|\frac{|f_\alpha(x+h)\varphi(x+h) - f_\alpha(x)\varphi(x)|}{|h|^{\bar{\gamma} - \alpha}} \Big\|^p_{L^p(\mathcal{K}, \d x)} \frac{\d h}{|h|^d}\Big)^{\frac{1}{p}}
  $$
  with $\bar{\gamma} := \alpha + \bar{\beta}$. An application of triangle inequality shows that
  \begin{align*}
    \vertiii{f*\tilde{\varphi}}_{\bar{\gamma},p,p} &\leq \Big(\int_{h \in B(0,1)} \Big\|\frac{|f_\alpha(x+h)\varphi(x+h) - f_\alpha(x)\varphi(x+h)|}{|h|^{\bar{\gamma} - \alpha}} \Big\|^p_{L^p(\mathcal{K}, \d x)} \frac{\d h}{|h|^d}\Big)^{\frac{1}{p}} \\ 
    &\quad + \Big(\int_{h \in B(0,1)} \Big\|\frac{|f_\alpha(x)\varphi(x+h) - f_\alpha(x)\varphi(x)|}{|h|^{\bar{\gamma} - \alpha}} \Big\|^p_{L^p(\mathcal{K}, \d x)} \frac{\d h}{|h|^d}\Big)^{\frac{1}{p}}.
  \end{align*}
  The first term on the right-hand side is bounded by $|\varphi|_{\infty} \vertiii{f*\tilde{\varphi}}_{\gamma,p,p}$. For the second term, note that by the assumption that $\varphi \in \mathcal{C}^{\beta}$, one has for any $x \in \mathcal{K}$ that
  \begin{equation*}
    |\varphi(x+h) - \varphi(x)| \le |\varphi|_{\beta}|h|^{\beta}, 
  \end{equation*}
  and consequently that 
  \begin{align*}
    \Big(\int_{h \in B(0,1)} \Big\|\frac{|f_\alpha(x)\varphi(x+h) - f_\alpha(x)\varphi(x)|}{|h|^{\bar{\gamma} - \alpha}} \Big\|^p_{L^p(\mathcal{K},\d x)}& \frac{\d h}{|h|^d}\Big)^{\frac{1}{p}} \\
    &\leq |\varphi|_{\beta}\|f_\alpha\|_{L^p}\Big(\int_{h \in B(0,1)} |h|^{p(\beta - \bar{\beta}) - d} \d h\Big)^{\frac{1}{p}}.
  \end{align*}
  Since $\beta - \bar{\beta} > 0$ and $p>1$, the integral $\int_{h \in B(0,1)} |h|^{p(\beta - \bar{\beta}) - d}\dd h$ can be uniformly bounded by a constant only depending on $\beta - \bar{\beta}$ for all $p > 1$. Therefore, we conclude that there is a constant $C(\varphi)$ only depending on the norm of $\varphi$ (and our choice of $\bar{\beta}$) such that for \textit{all} $p \in (1,\infty)$, the linear map $f \mapsto f*\tilde{\varphi}$ for $f \in \mathcal{D}^{\alpha+\bar{\beta}}_{p,p}(U_\alpha)$ is bounded with the operator norm bounded by $C(\varphi)$. 
  
  As a result, the SDE~\eqref{eq: RDE lifted to modelled distribution spaces} defined on $\mathcal{D}^{\alpha + \bar{\beta}}_{p,p}$, namely,
  \begin{equation*}
    \d Z^p_t = (Z^p_t * \tilde{\phi}) \dd W_t, \quad Z^p_0 = \Xi,
  \end{equation*}
  possesses a unique solution $Z^p_t \in \mathcal{D}^{\alpha + \bar{\beta}}_{p,p}$ defined on $[0,T]$ for any $p \in (1,\infty)$, by standard argument from the theory of SDEs. By Proposition~5.1 in \cite{Friz2014} we further conclude that $Z^p_t$ is also the unique global solution to the RRDE~\eqref{eq: RDE lifted to modelled distribution spaces} with respect to the It\^o lift. 
  
  Now we choose $p\in(1,\infty)$ large enough such that $\alpha + \bar{\beta} - 1/p > 0$. Since $\mathcal{D}^{\alpha + \bar{\beta}}_{p,p}(U_\alpha) \simeq \mathcal{B}^{\bar{\beta}}_{p,p}$, a classical Besov embedding argument shows that $\mathcal{D}^{\alpha + \bar{\beta}}_{p,p}(U_\alpha) \subset \mathcal{D}^{\alpha + \bar{\beta} - 1/p}_{\infty,\infty}(U_\alpha) \simeq \mathcal{C}^{\bar{\beta} - 1/p}$ continuously, we can consider $Z^p$ as the unique global solution of the RDE~\eqref{eq: RDE lifted to modelled distribution spaces} defined on $\mathcal{D}^{\alpha + \bar{\beta} - 1/p}_{\infty,\infty}(U_\alpha)$. Let $Y^p_t := \mathcal{R}^{\xi}(Z^p_t)$ for all $t \in [0,T]$. Then the same argument as we established in Step~2 with $\gamma =  \alpha + \bar{\beta} - 1/p > 0$ gives that $Y^p_t$ is a continuous curve in $\mathcal{C}^\alpha$ such that it solves equation~\eqref{eq: distribution-valued SDE} (with respect to the It\^o lift), i.e.,
  \begin{equation*}
    Y^p_t = \int_0^t (Y^p_t \cdot \phi) \dd W_t, \quad Y^p_0 = \xi, \quad t\in [0,T].
  \end{equation*}
  Moreover, since the multiplication with $\varphi$ is a bounded linear vector field, we can apply the continuity of It\^o--Lyons map (see e.g. Chapter~8 in \cite{Friz2014}) to deduce that the RDE~\eqref{eq: distribution-valued SDE} admits a unique solution (as long as it exists). This implies that the solution $Y^p$ actually does not depend on the choice of~$p$, and the proof is completed.
\end{proof}

\begin{remark}
  The SDE~\eqref{eq: distribution-valued SDE} with the linear vector field $V(Y_t):=Y_t\cdot \phi$ can be considered as one toy example where it has advantages to work with a mixed stochastic integration--regularity structure approach. Other examples of distribution-valued It\^o stochastic differential equations, which could be treated similarly to \eqref{eq: distribution-valued SDE} and the ''mixed`` approach has its advantage, are stochastic differential equations (on the plane) with vector fields like
  \begin{equation*}
     V(Y_t):= \int_0^\cdot g(Y_t(u)) \dd \mathbf{X}_u\quad \text{or} \quad  V(Y_t):= \tilde Y_t,
  \end{equation*}
  where $\tilde Y$ is the unique solution to $ \tilde Y_t(\cdot) := Y_t(\cdot) + \int_0^\cdot g(\tilde Y_t(u)) \dd \mathbf{X}_u$, for a given rough path $\mathbf{X}$ and a vector field~$g$ with suitable regularities.
\end{remark}

\subsection*{Acknowledgment} 

The authors thank Martin Hairer and Cyril Labb\'e for sharing their impressive ideas on embedding theorems for Besov spaces of modelled distributions at a pre-preprint stage of their work~\cite{Hairer2017} with us. D.J.P and J.T. gratefully acknowledge financial support of the Swiss National Foundation under Grant No.~$200021\_163014$. D.J.P. was affiliated to ETH Z\"urich when this project was initiated.


\vspace{0.5cm}
\noindent Chong Liu, Eidgen\"ossische Technische Hochschule Z\"urich, Switzerland\\ 
{\small \textit{E-mail address:} chong.liu@math.ethz.ch}\bigskip

\vspace{-.1cm}
\noindent David J. Pr\"omel, University of Oxford, United Kingdom\\
{\small\textit{E-mail address:} proemel@maths.ox.ac.uk}\bigskip

\vspace{-.1cm}
\noindent Josef Teichmann, Eidgen\"ossische Technische Hochschule Z\"urich, Switzerland\\
{\small\textit{E-mail address:} josef.teichmann@math.ethz.ch}\bigskip

\end{document}